\documentclass[preprint,11pt,number]{elsarticle}

\usepackage{amssymb}
\usepackage{amsmath}
\usepackage{amsthm}
\usepackage{mathtools}
\usepackage{hyperref}
\usepackage{float}
\usepackage{pgfplots}
\usepackage{booktabs}
\usepackage{comment}
\usepackage[margin=3.5cm]{geometry}

\newcommand{\N}{\mathbb{N}}
\newcommand{\R}{\mathbb{R}}
\newcommand{\C}{\mathbb{C}}
\newcommand{\V}{\mathbb{V}}
\newcommand{\Pro}{\mathbb{P}}
\newcommand{\Sp}{\mathbb{S}}
\newcommand{\T}{\mathcal{T}}
\newcommand{\St}{\mathcal{S}}
\newcommand{\M}{\mathcal{M}}
\DeclareMathOperator{\Sy}{S}
\DeclareMathOperator{\Seg}{Seg}
\DeclareMathOperator{\U}{U}
\DeclareMathOperator{\Sym}{Sym}

\DeclareMathOperator{\vHD}{vHDdeg}
\DeclareMathOperator{\HD}{HDdeg}
\DeclareMathOperator{\vHDp}{vHDpol}
\DeclareMathOperator{\HDp}{HDpol}
\DeclareMathOperator{\ED}{EDdegree}

\newcommand{\bz}{\mathbf{z}}
\newcommand{\bw}{\mathbf{w}}
\newcommand{\bu}{\mathbf{u}}
\newcommand{\bv}{\mathbf{v}}
\newcommand{\bn}{\mathbf{n}}
\newcommand{\bm}{\mathbf{m}}
\DeclareMathOperator{\Tr}{Tr}
\DeclareMathOperator{\vect}{vec}

\DeclareMathOperator{\Id}{Id}
\DeclareMathOperator{\hdet}{hdet}
\DeclareMathOperator{\Vol}{Vol}
\DeclareMathOperator{\MV}{MV}
\DeclareMathOperator{\conv}{conv}
\DeclareMathOperator{\disc}{disc}
\DeclareMathOperator{\sing}{sing}

\DeclareMathOperator{\rk}{rk}

\theoremstyle{definition}
\newtheorem{definition}{Definition}[section]
\newtheorem{example}{Example}[section]
\newtheorem{remark}{Remark}[section]
\theoremstyle{plain}
\newtheorem{proposition}{Proposition}[section]
\newtheorem{corollary}[proposition]{Corollary}
\newtheorem{theorem}[proposition]{Theorem}
\newtheorem{lemma}[proposition]{Lemma}
\newtheorem*{theorem*}{Theorem}
\newtheorem*{proposition*}{Proposition}

\journal{....}
\pgfplotsset{compat=1.18}
\begin{document}

\begin{frontmatter}

\title{The Hermitian Distance degree of Tensor spaces}

\author[1]{Davide Furch\`{i}}%% Author name

\affiliation[1]{organization={Dipartimento di Scienza e Alta Tecnologia, Università dell'Insubria},%Department and Organization
            %addressline={}, 
            city={Como},
            postcode={22100},
            country={Italy}}

%% Abstract
\begin{abstract}
In this paper, we investigate the Hermitian distance minimization problem for determinantal varieties, the Segre variety, and the Veronese variety. In particular, for binary forms, we obtain upper and lower bounds for the number of critical points that depend linearly on the order, and we determine all possible values in the case of order three.
\end{abstract}


\begin{comment}
\begin{highlights}
\item We define a number that evaluate the complexity of finding the minimum distance point from an algebraic variety
\item We develop several machinery address the problem
\item We present various examples
\end{highlights}
\end{comment}

%% Keywords
\begin{keyword}
Hermitian distance \sep Segre variety \sep Veronese variety \sep geometric measure of entanglement


%\MSC 12D10 \sep 14C17 \sep 15A54 \MSC[2020] 13P15
\MSC  13P25 \sep 14M12 \sep 15A69 \sep 30G30

\end{keyword}

\end{frontmatter}



\section{Introduction}\label{sec:intro}

Tensor varieties are of central importance in various fields of both mathematics and physics. In a physical setting the problem of finding the closest rank one tensor with respect to the Hermitian distance is referred to as the \emph{geometric measure of entanglement} and usually what is required is to maximize the \emph{injective tensor norm}.

In this paper we study this problem through an algebraic geometric perspective using the Hermitian Distance degree introduced in \cite{f2025}. We start from the case of determinantal varieties and then proceed to consider the Segre and Veronese varieties focusing on the case of $2\times 2\times 2$ tensors. For this examples we compute the values related to Hermitian distance.

\section{Preliminaries}\label{sec:pre}

Let $\V$ be a $n$-dimensional complex vector space with complex conjugation, that is an antilinear map $\bar{\phantom{v}}\colon\V\to\V$ such that $\bar{\phantom{v}}^2=\Id$, endowed with a Hermitian form $q$. We are interested in the critical points of the induced real valued function
\begin{align*}
    q_{\bu}\colon X&\subseteq\V\to\R\\
    &\bz\quad\longmapsto q(\bu-\bz,\bu-\bz)
\end{align*}
where $\bu\in\V$ and $X=V(f_1,\ldots,f_s)\subseteq\V$ is an algebraic variety with $f_1,\ldots,f_s\in\C[\bz]$.

We report briefly the reasoning in \cite{f2025}. Set a basis for which $q$ is the canonical Hermitian inner product, by \cite[Lemma 2.1]{f2025} a regular \emph{critical point} $\bz\in X$ of $q_{\bu}$ must satisfy $\bar{\bu}-\bar{\bz}\perp_{\R}T_{\bz}X$. Write $J(f)$ for the $s\times n$ Jacobian matrix of $(f_1,\ldots,f_s)$ and denote the ideals
\begin{equation*}
  I_{X}=\langle f_1,\ldots,f_{s}\rangle\quad\text{and}\quad I_{X_{\sing}}\coloneqq I_X+\langle\text{$c$-minors of $J(f)$}\rangle
\end{equation*}
where $c$ is the codimension of $X$. We introduce two new collections $\bw$ and $\bv$ of variables $\lbrace w_1,\ldots,w_n\rbrace$ and $\lbrace v_1,\ldots,v_n\rbrace$ respectively and extend the conjugate map to a map $\ast$ between polynomial rings
\begin{align*}
    \ast\colon\C[\bz,\bw&,\bu,\bv]\to\C[\bz,\bw,\bu,\bv]\\
    &g\qquad\mapsto\quad\ast(g)=g^{\ast}
\end{align*}
such that $z_k^{\ast}=w_k$, $w_k^{\ast}=z_k$, $u_k^{\ast}=v_k$, $v_k^{\ast}=u_k$ for $k=1,\ldots,n$ and $a^{\ast}=\bar{a}$ for any $a\in\C$. Thus, using the ideal
\begin{equation*}
  I_{X}^{\prime}\coloneqq\left\langle\text{$(c+1)$-minors of $\begin{bmatrix}
              \bv-\bw\\
              J_{\bz}(f)
  \end{bmatrix}$}\right\rangle,
\end{equation*}
we define the \emph{Hermitian critical ideal} of $X$ as the following saturation
\begin{equation}\label{cond}
  \left(I_{X}+(I_{X})^{\ast}+I_{X}^{\prime}+(I_{X}^{\prime})^{\ast}\right)\colon\left( I_{X_{\sing}}\cdot (I_{X_{\sing}})^{\ast}\right)^{\infty}\subseteq\C[\bz,\bw,\bu,\bv].
\end{equation}
For fixed $\bu$ and $\bv$ we will call the ideal above, in the polynomial ring $\C[\bz,\bw]$, the Hermitian critical ideal of $(\bu,\bv)$. Thanks to \cite[Lemma 3.1]{f2025} we are able to state the following definition.

\begin{definition}
Let $X$ be an algebraic variety. The \emph{virtual Hermitian Distance degree} of $X$ is the constant number of solutions of the Hermitian critical ideal \eqref{cond} of $(\bu,\bv)$ in a Zariski-open subset of $\V^2$, and it is denoted $\vHD(X)$. The \emph{Hermitian Distance degree} of $X$ is the subset of $\N$ consisting in the numbers of critical points of $q_{\bu}$ which do not vary for $\bu$ in an Euclidean-open subset of $\V$, and it is denoted $\HD(X)$.
\end{definition}

We will use results from \cite{f2024} to deal with polynomials that can have conjugation applied to the variables, we call them \emph{generalized polynomials}. Moreover, we will refer to objects such as the \emph{(virtual) Hermitian Distance discriminant}, (v)HD discriminant from \cite[Section 4]{f2025} and the \emph{(virtual) Hermitian Distance polynomial}, (v)HD polynomial from \cite[Section 5]{f2025}.

For the sake of simplicity, we often will use the notation $[n]$ to indicate the set of natural numbers from $1$ up to $n\in\N$.

%----------------------------------------------------------------------------------------------------------------------------------------------

\section{Determinantal varieties}\label{sec:matrices}\thispagestyle{plain}

We identify the vector space $\V$ with the space of matrices $\M_{n}^{m}\coloneqq\C^n\otimes\C^m$ or simply $\M_n$ if $n=m$. For the Euclidean case see \cite{oss2014,dhost2014}.

Rewrite the Hermitian inner product as
\begin{equation*}
    \langle A,B\rangle_{\C}=\sum_{j_1=1}^{n}\sum_{j_2=1}^{m}A_{j_1j_2}\bar{B}_{j_1j_2}=\Tr(AB^H),
\end{equation*}
the norm $\|\phantom{A}\|_{\C}$ is usually called the \emph{Frobenius norm} or the \emph{Hilbert–Schmidt norm}.

\begin{definition}\label{defnormmat}
	Let $\|\phantom{\bz}\|_2$ be the $2$-norm on $\C^n$ and $\C^m$, then the \emph{induced matrix norm} of $\M_n^m$ is defined to be
	\begin{equation*}
		\|A\|_2\coloneqq\sup_{\bz\in\C^m\setminus\lbrace 0\rbrace}\frac{\|A\bz\|_2}{\|\bz\|_2}=\max_{\substack{\bz\in\C^m \\ \|\bz\|_2=1}}\|A\bz\|_2\qquad\text{for $A\in\M_n^m$.}
	\end{equation*}
\end{definition}

\begin{remark}\label{matrixnorm}
The Hermitian norm $\|\phantom{A}\|_{\C}$ and the induced norm $\|\phantom{A}\|_{2}$ do not coincide in $\M_{n}^{m}$. In general, if $\sigma_1\geq\ldots\geq\sigma_{\min\lbrace n,m\rbrace}$ are the singular values of $A\in\M_{n}^{m}$ then  
\begin{equation*}
    \|A\|_{\C}^2=\sum_{k=1}^{\min\lbrace n,m\rbrace}\sigma_k^2\qquad\text{and}\qquad\|A\|_{2}^2=\sigma_1^2.
\end{equation*}
\end{remark}

%------------------------------------------------------------------------------------------

\subsection{General matrices}\label{ssec:genmat}
We investigate here $\rk_{r}(\M_{n}^{m})$ the subvariety of $\M_{n}$ of matrices with rank bounded by some $r\in[\min\lbrace n,m\rbrace]$. We start by recalling the Eckart-Young-Mirsky theorem. This result solves the Hermitian distance problem for determinantal varieties, we do not report its proof, however we will see how to recover it from Corollary~\ref{coroeym}.

\begin{theorem*}[Complex Eckart-Young-Mirsky]
Let $A\in\M_{n}^{m}$ with $A=U\Sigma V^{H}$ a SVD and let $r\in[\rk(A)]$. All the critical points of the Hermitian distance function from $A$ to the variety $\rk_{r}(\M_{n}^{m})\setminus\rk_{r-1}(\M_{n}^{m})$
are $U(\Sigma_{k_1}+\ldots+\Sigma_{k_r})V^{H}$ with $1\leq k_1<\ldots<k_r\leq\rk(A)$, where $\Sigma_{k}\in\M_{n}^{m}$ is everywhere zero apart from the $(k,k)$ entry that is the $k$-th singular value for $k\in[\rk(A)]$.\label{cEY}
\end{theorem*}

Set $X=\rk_{r}(\M_{n}^{m})$ in this subsection. From the theorem above we get $\HD(X)=\lbrace\binom{\min\lbrace n,m\rbrace}{r}\rbrace$. We now compute the number of solutions of the Hermitian critical ideal.

\begin{lemma}\label{matrices}
Let $\bu,\bv\in\M_{n}^{m}$ be such that $\rk(\bu\bv^T)\geq r$, the number of solutions of the Hermitian critical ideal of $(\bu,\bv)$ is $\binom{\rk(\bu\bv^T)}{r}$. In particular, $\vHD(X)=\binom{\min\lbrace n,m\rbrace}{r}$.
\end{lemma}
\begin{proof}
Consider a point $(\bz,\bw)\in X\times\overline{X}$ and the formulations
\begin{equation*}
   \bz=\sum_{k=1}^r\lambda_{\bz}^{(k)}\bz^{(k,1)}\otimes\bz^{(k,2)}\qquad\text{and}\qquad\bw=\sum_{j=1}^r\lambda_{\bw}^{(j)}\bw^{(j,1)}\otimes\bw^{(j,2)}
\end{equation*}
where $\lbrace\bz^{(k,1)}\rbrace_{k=1}^r,\lbrace\bw^{(j,1)}\rbrace_{j=1}^r\subseteq\C^n$, $\lbrace\bz^{(k,2)}\rbrace_{k=1}^r,\lbrace\bw^{(j,2)}\rbrace_{j=1}^r\subseteq\C^m$ and $0\neq\lambda_{\bz}^{(k)},\lambda_{\bw}^{(j)}\in\C$ for any $k,j=1,\ldots,r$, to write
\begin{align*}
    \langle\bz-\bu,\bw-\bv\rangle_{\R}&=\sum_{k=1}^r\sum_{j=1}^r\lambda_{\bz}^{(k)}\lambda_{\bw}^{(j)}\langle\bz^{(k,1)},\bw^{(j,1)}\rangle_\R\langle\bz^{(k,2)},\bw^{(j,2)}\rangle_\R\\
    &\quad-\sum_{k=1}^r\lambda_{\bz}^{(k)}\langle\bz^{(k,1)}\otimes\bz^{(k,2)},\bv\rangle_{\R}-\sum_{j=1}^r\lambda_{\bw}^{(j)}\langle\bu,\bw^{(j,1)}\otimes\bw^{(j,2)}\rangle_{\R}+\langle\bu,\bv\rangle_{\R}.
\end{align*}
The derivative with respect to $z_{\ell}^{(k,1)}$ for $\ell=1,\ldots,n$ is
\begin{equation*}
    \lambda_{\bz}^{(k)}\left(\sum_{j=1}^r\lambda_{\bw}^{(j)}\bw^{(j,1)}_{\ell}\langle\bz^{(k,2)},\bw^{(j,2)}\rangle_\R-\langle e_{\ell}\otimes\bz^{(k,2)},\bv\rangle_{\R}\right)
\end{equation*}
and we want to study the zero locus. By considering the derivatives with respect to any $z_{\ell}^{(k,s)}$ and $w_{\ell}^{(j,s)}$ for $s=1,2$, we get the conditions
\begin{alignat}{2}
    \bu^T\bw^{(j,1)}&=\sum_{k=1}^r\lambda_{\bz}^{(k)}\langle\bz^{(k,1)},\bw^{(j,1)}\rangle_{\R}\bz^{(k,2)},\hspace{0.6cm}\bu\bw^{(j,2)}&&=\sum_{k=1}^r\lambda_{\bz}^{(k)}\langle\bz^{(k,2)},\bw^{(j,2)}\rangle_{\R}\bz^{(k,1)},\label{Eigen1}\\
    \bv^T\bz^{(k,1)}&=\sum_{j=1}^r\lambda_{\bw}^{(j)}\langle\bz^{(k,1)},\bw^{(j,1)}\rangle_{\R}\bw^{(j,2)},\hspace{0.5cm}\bv\bz^{(k,2)}&&=\sum_{j=1}^r\lambda_{\bw}^{(j)}\langle\bz^{(k,2)},\bw^{(j,2)}\rangle_{\R}\bw^{(j,1)}.\label{Eigen2}
\end{alignat}
From equations \eqref{Eigen2}, by knowing $\lbrace\bz^{(k,1)}\rbrace_{k=1}^r$, we are able to choose $\lbrace\bw^{(j,2)}\rbrace_{j=1}^r$ basis of the subspace given by the images under $\bv^T$ and similarly for the basis $\lbrace\bw^{(j,1)}\rbrace_{j=1}^r$. There is no essential difference in the choice of these bases. On the other hand, by applying the maps $\bu$ and $\bu^T$ respectively we get
\begin{align*}
  \bu\bv^T\bz^{(\ell,1)}&=\sum_{k=1}^r\left(\sum_{j=1}^r\lambda_{\bz}^{(k)}\lambda_{\bw}^{(j)}\langle\bz^{(k,1)},\bw^{(j,1)}\rangle_\R\langle\bz^{(k,2)},\bw^{(j,2)}\rangle_\R\right)\bz^{(k,1)},\\
  \qquad\bu^T\bv\bz^{(\ell,2)}&=\sum_{k=1}^r\left(\sum_{j=1}^r\lambda_{\bz}^{(k)}\lambda_{\bw}^{(j)}\langle\bz^{(k,1)},\bw^{(j,1)}\rangle_\R\langle\bz^{(k,2)},\bw^{(j,2)}\rangle_\R\right)\bz^{(k,2)}
\end{align*}
for $\ell=1,\ldots,r$. Thus, by knowing $\lbrace\bz^{(k,1)}\rbrace_{k=1}^r$ there is no essentially different choice for the basis $\lbrace\bz^{(k,2)}\rbrace_{k=1}^r$. In particular, we need to choose a collection $\lbrace\bz^{(k,1)}\rbrace_{k=1}^r$ which generates a $r$ dimensional eigenspace of $\bu\bv^T$. Multiplying on the left the left equations of \eqref{Eigen1} and \eqref{Eigen2} by $(\bw^{(j,2)})^T$ and $(\bz^{(k,2)})^T$ respectively, we get the equalities
\begin{align*}
    \langle\bz^{(k,1)}\otimes\bz^{(k,2)},\bv\rangle_{\R}&=\sum_{j=1}^r\lambda_{\bw}^{(j)}\langle\bz^{(k,1)},\bw^{(j,1)}\rangle_\R\langle\bz^{(k,2)},\bw^{(j,2)}\rangle_{\R},\\
    \langle\bu,\bw^{(j,1)}\otimes\bw^{(j,2)}\rangle_{\R}&=\sum_{k=1}^r\lambda_{\bz}^{(k)}\langle\bz^{(k,1)},\bw^{(k,1)}\rangle_\R\langle\bz^{(k,2)},\bw^{(k,2)}\rangle_{\R},
\end{align*}
and we compute the scalars $\lambda_{\bz}^{(k)},\lambda_{\bw}^{(j)}$ for $k,j=1,\ldots,r$ by solving the linear systems provided by these last equations. At the end, we have $\binom{\rk(\bu\bv^T)}{r}$ essentially different possible choices for the collection $\lbrace\bz^{(k,1)}\rbrace_{k=1}^r$ of eigenvectors of $\bu\bv^T$ which in turn permits to compute the other vectors and the scalars.
\end{proof}

\begin{proposition}\label{hddegmatrix}
The roots in $t^2$ of the vHD polynomial $\vHDp_{X,\bu,\bv}(t^2)$ are the sums of $\min\lbrace n,m\rbrace-r$ eigenvalues of $\bu\bv^T$ if $n\leq m$ or of $\bu^T\bv$ if $n\geq m$.    
\end{proposition}
\begin{proof}
Note that from the proof of Lemma~\ref{matrices} we get that
\begin{equation*}
    \langle\bz,\bv\rangle_{\R}=\langle\bu,\bw\rangle_{\R}=\sum_{k=1}^r\sum_{j=1}^r\lambda_{\bz}^{(k)}\lambda_{\bw}^{(j)}\langle\bz^{(k,1)},\bw^{(j,1)}\rangle_\R\langle\bz^{(k,2)},\bw^{(j,2)}\rangle_\R=\langle\bz,\bw\rangle_{\R}.
\end{equation*}
is the sum of $r$ non zero eigenvalues of $\bu\bv^T$. Let $\lambda_1,\ldots,\lambda_{\min\lbrace n,m\rbrace}$ be the eigenvalues of $\bu\bv^{T}$ or $\bu^T\bv$. Then, if $(\bz,\bw)$ is a solution associated to the eigenvalues $\lambda_{k_1},\ldots,\lambda_{k_r}$ for different $k_1,\ldots,k_r\in[\min\lbrace n,m\rbrace]$, the squared Hermitian distance function takes the value
\begin{equation*}
  \langle\bz,\bw\rangle_{\R}-\langle\bz,\bv\rangle_{\R}-\langle\bu,\bw\rangle_{\R}+\langle\bu,\bv\rangle_{\R}=\sum_{k\in W}\lambda_k-\sum_{k\in W}\lambda_k-\sum_{k\in W}\lambda_k+\sum_{j=1}^{\min\lbrace n,m\rbrace}\lambda_j=\sum_{k\notin W}\lambda_k
\end{equation*}
where $W=\lbrace k_1,\ldots,k_r\rbrace\subseteq[\min\lbrace n,m\rbrace]$ is a subset of cardinality $r$. The claim follows.
\end{proof}

\begin{example}\label{matrix2x}
Let $X=\rk_{1}(\M_{2})\subseteq\M_2$, we show how to proceed to compute the solution of the Hermitian critical ideal of $(\bu,\bv)$ where
\begin{equation*}
  \bu=\begin{bmatrix}
        u_{11} & u_{12}\\
        u_{21} & u_{22}
    \end{bmatrix}\in\M_{2}\quad\text{and}\quad\bv=\begin{bmatrix}
        v_{11} & v_{12}\\
        v_{21} & v_{22}
    \end{bmatrix}\in\M_{2}.
\end{equation*}
We consider the two matrices 
\begin{equation*}
    \bu\bv^T=\begin{bmatrix}
    \bu_{1,\cdot}\bv_{1,\cdot}^T & \bu_{1,\cdot}\bv_{2,\cdot}^T\\
    \bu_{2,\cdot}\bv_{1,\cdot}^T & \bu_{2,\cdot}\bv_{2,\cdot}^T
\end{bmatrix}\in\M_2\quad\text{and}\quad\bu^T\bv=\begin{bmatrix}
    \bu_{\cdot,1}^T\bv_{\cdot,1} & \bu_{\cdot,1}^T\bv_{\cdot,2}\\
    \bu_{\cdot,2}^T\bv_{\cdot,1} & \bu_{\cdot,2}^T\bv_{\cdot,2}\\
\end{bmatrix}\in\M_2.
\end{equation*}
These matrices have two equal eigenvalues
\begin{equation*}
    \lambda^{\pm}=\frac{\Tr(\bu\bv^{T})\pm\sqrt{(\bu_{1,\cdot}\bv^T_{1,\cdot}-\bu_{2,\cdot}\bv_{2,\cdot}^T)^2+4\bu_{1,\cdot}\bv_{2,\cdot}^T\bu_{2,\cdot}\bv_{1,\cdot}^T}}{2}
\end{equation*}
with associate eigenvectors
\begin{equation*}
    \bz^{(\pm,1)}=\begin{bmatrix}
        \lambda^{\pm}-2\bu_{2,\cdot}\bv_{2,\cdot}^T\\
        2\bu_{2,\cdot}\bv_{1,\cdot}^T
    \end{bmatrix}\in\C^2\quad\text{and}\quad\bz^{(\pm,2)}=\begin{bmatrix}
        \lambda^{\pm}-2\bu_{\cdot,2}^T\bv_{\cdot,2}\\
        2\bu_{\cdot,2}^T\bv_{\cdot,1}
    \end{bmatrix}\in\C^2
\end{equation*}
respectively. Then we compute the vectors $\bw^{(\pm,1)}=\bv\bz^{(\pm,2)}\in\C^2$ and $\bw^{(\pm,2)}=\bv^T\bz^{(\pm,1)}\in\C^2$ and the scalars
\begin{equation*}
    \lambda_{\bz}^{\pm}=\frac{(\bw^{(\pm,2)})^T\bu\bw^{(1),\pm}}{(\bz^{(\pm,1)})^T\bw^{(\pm,1)}(\bz^{(\pm,2)})^T\bw^{(\pm,2)}}\qquad\text{and}\qquad\lambda_{\bw}^{\pm}=\frac{(\bz^{(\pm,2)})^T\bv\bz^{(1),\pm}}{(\bz^{(\pm,1)})^T\bw^{(\pm,1)}(\bz^{(\pm,2)})^T\bw^{(\pm,2)}}
\end{equation*}
to get the four matrices $\lambda_{\bz}^{\pm}\bz^{(\pm,1)}\otimes\bz^{(\pm,2)}\in\M_2$ and $\lambda_{\bw}^{\pm}\bw^{(\pm,2)}\otimes\bw^{(\pm,2)}\in\M_2$ which yields the two solutions of the Hermitian critical ideal
\begin{equation*}
    \left(\lambda_{\bz}^{+}\bz^{(+,1)}\otimes\bz^{(+,2)},\lambda_{\bw}^{+}\bw^{(+,1)}\otimes\bw^{(+,2)}\right)\quad\text{and}\quad\left(\lambda_{\bz}^{-}\bz^{(-,1)}\otimes\bz^{(-,2)},\lambda_{\bw}^{-}\bw^{(-,1)}\otimes\bw^{(-,2)}\right).
\end{equation*}
\end{example}

\begin{corollary}\label{rmatrices}
The zero locus of the polynomial discriminant $\Delta_{t^2}\HDp_{X,\bu}(t^2)$ defines the subset of $\M_{n}^{m}$ of matrices that admit two different sums of $\min\lbrace n,m\rbrace-r$ of their singular values to have the same result. If $r=\min\lbrace n,m\rbrace-1$ then
\begin{equation*}
    \begin{cases}
        \vHDp_{X,\bu,\bv}(t^2)=\det(t^2I_n-\bu\bv^{T}) &\text{when $n\leq m$,}\\
        \vHDp_{X,\bu,\bv}(t^2)=\det(t^2I_m-\bv^T\bu) & \text{when $n\geq m$,}
    \end{cases}
\end{equation*}
while if $r=1$ then
\begin{equation*}
    \begin{cases}
        \vHDp_{X,\bu,\bv}(t^2)=\det\left((\Tr(\bu\bv^T)-t^2)I_n-\bv\bu^{T}\right) &\text{when $n\leq m$,}\\
        \vHDp_{X,\bu,\bv}(t^2)=\det\left((\Tr(\bu\bv^T)-t^2)I_m-\bu\bv^{T}\right) & \text{when $n\geq m$.}
    \end{cases}
\end{equation*}
In general, if $\lambda_1,\ldots,\lambda_{\min\lbrace n,m\rbrace}$ are the eigenvalues of $\bu\bv^{T}$ if $n\leq m$ or of $\bv\bu^T$ if $n\geq m$, then 
\begin{equation*}
    \vHDp_{X,\bu,\bv}(t^2)=\prod\limits_{\text{$W\subseteq[\min\lbrace n,m\rbrace]$ of cardinality $\min\lbrace n,m\rbrace-r$}}\left(t^2-\sum_{k\notin W}\lambda_k\right).
\end{equation*}
\end{corollary}

\begin{example}\label{2matrix2x}
Let $X=\rk_{1}(\M_{2}^{m})\subseteq\M_2^{m}$, then
\begin{align*}
    \vHDp_{X,\bu,\bv}(t^2)=\det(t^2I_2-\bu\bv^T)=t^4-\langle\bu,\bv\rangle_{\R} t^2+\det(\bu\bv^T).
\end{align*}
The HD polynomial then is
\begin{equation*}
  \HDp_{X,\bu}(t^2)=t^4-\|\bu\|_{\C}^2t^2+\det(\bu\bu^H)
\end{equation*}
and since $\Delta_{t^2}\HDp_{X,\bu}(t^2)=\|\bu\|_{\C}^4-4\det(\bu\bu^H)\geq 0$ by the inequality of arithmetic and geometric means, we obtain two critical points if and only if $\det(\bu\bu^H)\neq 0$ or equivalently $\bu$ has two positive singular values.

Let now $m=2$, the vHD discriminant is defined by the polynomial
\begin{equation*}
    \Delta_{t^2}\vHDp_{X,\bu,\bv}(t^2)=\langle\bu,\bv\rangle_{\R}^2-4|\det(\bu\bv^T)|.
\end{equation*}
Setting $\bv=\bar{\bu}$, the inequality $\det(\bu\bu^H)\neq 0$ is equivalent to $\det(\bu)\neq 0$ and the two critical points have same distance if and only if the generalized polynomial defining the HD discriminant
\begin{align*}
\Delta_{t^2}\HDp_{X,\bu}(t^2)&=\left(\|\bu\|_{\C}^2+2|\det(\bu)|\right)\left(\|\bu\|_{\C}^2-2|\det(\bu)|\right)
\end{align*}
vanishes. If $\bu$ is real it becomes 
\begin{equation*}
\Delta_{t^2}\HDp_{X,\bu}(t^2)=\left((u_{11}+u_{22})^2+(u_{12}-u_{21})^2\right)\left((u_{11}-u_{22})^2+(u_{12}+u_{21})^2\right).
\end{equation*}
\end{example}

\begin{remark}
Note that, in Example~\ref{2matrix2x}, the generalized polynomial defining the HD discriminant is non negative. The fact that this polynomial does not change sign could have been anticipated by Proposition~\ref{matrices} considering \cite[Corollary 4.3]{f2025}. Also, the discriminant of the HD polynomial is non negative. We will see this is not a special case in Subsection~\ref{ssec:parl}.
\end{remark}

%------------------------------------------------------------------------------------------

\subsection{Symmetric matrices}
We investigate here $\rk_{r}(\Sy_{n})$ the subvariety of $\M_{n}$ of symmetric matrices with rank bounded by some $r\in[\min\lbrace n,m\rbrace]$. 

Set $X=\rk_{r}(\Sy_{n})$ in this subsection. Similarly to the general case, it holds the equality $\HD(X)=\lbrace\binom{n}{r}\rbrace$. We now compute the number of solutions of the Hermitian critical ideal.

\begin{lemma}\label{symmatrices}
Let $\bu,\bv\in\M_{n}$ be such that $\rk((\bu+\bu^T)(\bv+\bv^T))\geq r$, the number of solutions of the Hermitian critical ideal of $(\bu,\bv)$ is $\binom{\rk((\bu+\bu^T)(\bv+\bv^T))}{r}$. In particular, $\vHD(X)=\binom{n}{r}$.
\end{lemma}
\begin{proof}
The proof is analogous to the proof Lemma~\ref{matrices}, we report the fundamental steps for clarity. Consider a point $(\bz,\bw)\in X\times\overline{X}$ and the formulations
\begin{equation*}
   \bz=\sum_{k=1}^r\lambda_{\bz}^{(k)}(\bz^{(k)})^{\otimes 2}\qquad\text{and}\qquad\bw=\sum_{j=1}^r\lambda_{\bw}^{(j)}(\bw^{(j)})^{\otimes 2}
\end{equation*}
where $\lbrace\bz^{(k)}\rbrace_{k=1}^r,\lbrace\bw^{(j)}\rbrace_{j=1}^r\subseteq\C^n$ and $0\neq\lambda_{\bz}^{(k)},\lambda_{\bw}^{(j)}\in\C$ for $k,j=1,\ldots,r$. We rewrite rewrite the inner product $\langle\bz-\bu,\bw-\bv\rangle_{\R}$ and consider its derivatives with respect to any $z_{\ell}^{(k)}$ and $w_{\ell}^{(j)}$ to get the conditions
\begin{alignat}{2}
    (\bu+\bu^T)\bw^{(j)}&=2\sum_{k=1}^r\lambda_{\bz}^{(k)}\langle\bz^{(k)},\bw^{(j)}\rangle_{\R}\bz^{(k)},\notag\\
    (\bv+\bv^T)\bz^{(k)}&=2\sum_{j=1}^r\lambda_{\bw}^{(j)}\langle\bz^{(k)},\bw^{(j)}\rangle_{\R}\bw^{(j)}.\label{Eigen2.1}
\end{alignat}
From equations \eqref{Eigen2.1}, by knowing $\lbrace\bz^{(k)}\rbrace_{k=1}^r$, we are able to choose $\lbrace\bw^{(j)}\rbrace_{j=1}^r$ basis of the subspace given by the images under $\bv+\bv^T$. There is no essential difference in the choice of this basis. On the other hand, by applying the map $\bu+\bu^T$ to equations \eqref{Eigen2.1}, we get
\begin{equation*}
  (\bu+\bu^T)(\bv+\bv^T)\bz^{(\ell)}=4\sum_{k=1}^r\left(\sum_{j=1}^r\lambda_{\bz}^{(k)}\lambda_{\bw}^{(j)}\langle\bz^{(k)},\bw^{(j)}\rangle_\R^2\right)\bz^{(k)},
\end{equation*}
for $\ell=1,\ldots,r$. Thus, we need to choose a collection $\lbrace\bz^{(k)}\rbrace_{k=1}^r$ which generates a $r$ dimensional eigenspace of $(\bu+\bu^T)(\bv+\bv^T)/4$. Multiplying on the left the equations \eqref{Eigen1} and \eqref{Eigen2} by $(\bw^{(j)})^T$ and $(\bz^{(k)})^T$ respectively, we get the equalities
\begin{align*}
    \langle(\bz^{(k)})^{\otimes 2},\bv+\bv^T\rangle_{\R}&=2\sum_{j=1}^r\lambda_{\bw}^{(j)}\langle\bz^{(k)},\bw^{(j)}\rangle_\R^2\\
    \langle\bu+\bu^T,(\bw^{(j)})^{\otimes 2}\rangle_{\R}&=2\sum_{k=1}^r\lambda_{\bz}^{(k)}\langle\bz^{(k)},\bw^{(k)}\rangle_\R^2,
\end{align*}
and we compute the scalars $\lambda_{\bz}^{(k)},\lambda_{\bw}^{(j)}$ for $k,j=1,\ldots,r$ by solving the linear systems provided by these last equations. At the end, we have $\binom{\rk((\bu+\bu^T)(\bv+\bv^T))}{r}$ essentially different possible choices for the collection $\lbrace\bz^{(k)}\rbrace_{k=1}^r$ of eigenvectors of $(\bu+\bu^T)(\bv+\bv^T)/4$ which in turn permits to compute the other vectors and the scalars.
\end{proof}

We list a collection of results that are similar to the general case.

\begin{proposition}
The roots in $t^2$ of the vHD polynomial $\vHDp_{X,\bu,\bv}(t^2)$ are the sums of $n-r$ non zero eigenvalues of $(\bu+\bu^T)(\bv+\bv^T)/4$.
\end{proposition}

\begin{corollary}\label{rsmatrices}
The zero locus of the polynomial discriminant $\Delta_{t^2}\HDp_{X,\bu}(t^2)$ defines the subset of $\M_{n}$ of matrices $\bu$ such that the matrix $\bu+\bu^T$ admits two different sums of $n-r$ of its singular values to have the same result. If $r=n-1$ then
\begin{equation*}
    \vHDp_{X,\bu,\bv}(t^2)=\det\left(t^2I_n-\frac{(\bu+\bu^{T})(\bv+\bv^{T})}{4}\right)
\end{equation*}
while if $r=1$ then
\begin{equation*}
    \vHDp_{X,\bu,\bv}(t^2)=\det\left(\left(\Tr\left(\frac{(\bu+\bu^{T})(\bv+\bv^{T})}{4}\right)-t^2\right)I_n-\frac{(\bv+\bv^{T})(\bu+\bu^{T})}{4}\right)
\end{equation*}
In general, if $\lambda_1,\ldots,\lambda_{\min\lbrace n,m\rbrace}$ are the eigenvalues of $(\bu+\bu^{T})(\bv+\bv^{T})/4$, then 
\begin{equation*}
    \vHDp_{X,\bu,\bv}(t^2)=\prod\limits_{\text{$W\subseteq[n]$ of cardinality $n-r$}}\left(t^2-\sum_{k\notin W}\lambda_k\right).
\end{equation*}
\end{corollary}

\begin{corollary}
Let $\bu,\bv\in\Sy_n$, the zero locus of the Hermitian critical ideal of $\rk_{r}(\M_{n})$ of $(\bu,\bv)$ coincides with the zero locus of the Hermitian critical ideal of $\rk_{r}(\Sy_{n})$ of $(\bu,\bv)$. In particular, the critical points of the distance function of $\bu$ from $\rk_{r}(\M_{n})$ are all in $\rk_{r}(\Sy_{n})$, i.e.\ the critical points of a symmetric matrix are all symmetric. Moreover, it holds the equality $\vHDp_{\rk_r(\Sy_n),\bu,\bv}(t^2)=\vHDp_{\rk_r(\M_{n}),\bu,\bv}(t^2)$.
\end{corollary}
\begin{proof}
The claim follows from the proof of Lemma~\ref{symmatrices} by noting that if $\bu$ and $\bv$ are symmetric then $(\bu+\bu^T)(\bv+\bv^T)/4=\bu\bv^T$.
\end{proof}

\begin{example}
Let $X=\rk_{1}(\Sy_{2})=V(z_1z_4-z_2z_3,z_2-z_3)\subseteq\M_2$, alternatively to this definition, we can simplify our problem and consider the variety $\hat{X}_2\subseteq\C^3\simeq\Sy_2$ we will study in Example~\ref{exratnorm2} and the Hermitian form $q(\bz,\bw)=z_1\bar{w}_1+2z_2\bar{w}_2+z_3\bar{w}_3$. With this choice, the HD polynomial is
\begin{align*}
    \HDp_{\hat{X}_2,\bu,\bv}(t^2)=t^4-\langle\bu,\bv\rangle_{\R} t^2+u_2^2v_2^2+u_1v_1+u_3v_3-u_1v_2^2u_3-v_1u_2^2v_3
\end{align*}
and if we consider the point $\bu$ as the symmetric matrix $\begin{bmatrix}
    u_1 & u_2\\
    u_2 & u_3
\end{bmatrix}$ and similarly for $\bv$ we obtain the formula 
\begin{equation*}
    \HDp_{\hat{X}_2,\bu,\bv}(t^2)=\det(t^2I_2-\bu\bv^T)=t^4-\langle\bu,\bv\rangle_{\R}t^2+\det(\bu\bv^T)
\end{equation*}
which was known from Example~\ref{2matrix2x}.
\end{example}

%------------------------------------------------------------------------------------------

\subsection{Discriminant of a Hermitian matrix}\label{ssec:parl}

\begin{definition}
Let $A\in \M_{n}$ with possibly repeated eigenvalues $\lambda_1,\ldots,\lambda_n$. Its \emph{discriminant} is the product
\begin{equation*}
    \disc(A)\coloneqq\Delta_t\det(tI_n-A)=\prod_{k\neq j}(\lambda_k-\lambda_j)=(-1)^{\frac{n(n-1)}{2}}\prod_{1\leq k<j\leq n}(\lambda_k-\lambda_j)^2.
\end{equation*}
\end{definition}

In \cite[Section 2]{p2002} the author obtains the following formula for the discriminant. Let us indicate for $A\in\M_n^m$ and $\varphi=\lbrace\varphi_1,\ldots,\varphi_k\rbrace\subseteq[n]$ the matrix and vector 
\begin{equation*}
  A(\varphi,:)\coloneqq\begin{bmatrix}
                A_{\varphi_1,\cdot}\\
                \vdots\\
                A_{\varphi_k,\cdot}
               \end{bmatrix}\in\M_{k}^m\qquad\text{and}\qquad\vect(A)\coloneqq\begin{bmatrix}
                A_{1,\cdot}^T\\
                \vdots\\
                A_{n,\cdot}^T
               \end{bmatrix}\in\C^{nm}.
\end{equation*}

\begin{proposition*}[\textbf{Parlett '02}]
Let $A\in \M_{n}$ then
\begin{equation*}
  \disc(A)=\det(\mathcal{L}_A^T\mathcal{L}_A)=\sum_{\text{$\varphi\subseteq[n^2]$ of cardinality $n$}}\det(\mathcal{L}_{A^{T}}(\varphi,\colon))\det(\mathcal{L}_A(\varphi,\colon))
\end{equation*}
where $\mathcal{L}_A\in\M_{n^2}^{n}$ is the matrix $\mathcal{L}_A=\begin{bmatrix}
      \vect(I_n) & \vect(A) & \vect(A^2)& \cdots & \vect(A^{n-1})
    \end{bmatrix}$.
In particular, if $A$ is symmetric its discriminant can be written as a sum of squares.\label{parl}
\end{proposition*}

In the case of Hermitian matrices we get the following result.

\begin{corollary}\label{sqmod}
The discriminant of a Hermitian matrix $A\in \M_{n}$ can be written as a sum of squared modules.
\end{corollary}
\begin{proof}
The assertion follows from Proposition~\ref{parl} and the equalities $\mathcal{L}_{A^{T}}=\mathcal{L}_{\bar{A}}=\overline{\mathcal{L}_{A}}$ which imply $\det(\mathcal{L}_{A^{T}}(\varphi,\colon))=\overline{\det(\mathcal{L}_{A}(\varphi,\colon))}$.
\end{proof}

\begin{corollary}\label{coroeym}
Let $r\in[\min\lbrace n,m\rbrace]$, the discriminants of the HD polynomials of $\rk_r(\M_{n}^{m})$ and $\rk_r(\Sy_n)$ can be written as sums of squared modules. Moreover, the real codimension of the HD discriminant of both varieties is greater or equal than $2$.
\end{corollary}
\begin{proof}
Let $X=\rk_r(\M_{n}^{m})$ and assume without loss of generality $n\leq m$. If $\lambda_1,\ldots,\lambda_n$ are the eigenvalues of $\bu\bu^H$, the statement follows from the equalities
\begin{align*}
  \HDp_{X,\bu}(t^2)&=\prod\limits_{\text{$W\subseteq[n]$ of cardinality $n-r$}}\left(t^2-\sum_{k\notin W}\lambda_k\right)\\
  &=\det\left(t^2I_{\binom{n}{r}}-\begin{bmatrix}
      \sum_{k\notin W_1}\lambda_k & \cdots & 0\\
      \vdots & \ddots & \vdots\\
      0 & \cdots & \sum_{k\notin W_{\binom{n}{r}}}\lambda_k
  \end{bmatrix}\right)
\end{align*}
exploited in Corollary~\ref{rmatrices} where $W_j$ for $j=1,\ldots,\binom{n}{r}$ are the different subsets of $[n]$ of cardinality $n-r$ and Corollary~\ref{sqmod}. The claim follows analogously for $\rk_r(\Sy_n)$ by considering the Hermitian matrix $(\bu+\bu^T)(\bu^H+\bar{\bu})/4$. The last statement follows from the fact that the sum is non trivial in the coefficients of $\bu$.
\end{proof}

In particular, this last results confirms that the set $\HD$ contains only one value for both the determinantal varieties and, as a consequence, by the proof of Lemma~\ref{matrices}, we reobtain the complex Eckart-Young-Mirsky theorem reported in Subsection~\ref{ssec:genmat}.

%----------------------------------------------------------------------------------------------------------------------------------------------

\section{Tensors spaces}\label{sec:segre}\thispagestyle{plain}

We turn our attention to the varieties of tensors, the Euclidean distance case is treated in \cite{dhost2014,of2014,dh2016}. In particular, in \cite{of2014} it is shown a useful way to compute the $\ED$ of the Segre-Veronese variety, while in \cite{dh2016} the authors provide a formula to compute the average number of critical points for real tensors with respect to a probability measure.

Let $d\in\N$, $\bn=(n_1,\ldots,n_d)\in\N^{d}$ be a vector of indices and let $\V_1,\ldots,\V_d$ be complex vector spaces of dimension $n_k$ endowed with a Hermitian form $q_k$ respectively for $k=1,\ldots,d$. In this section we consider the vector space $\V$ to be the tensor product $\V_1\otimes\ldots\otimes\V_d$.

It is well known that there exists a unique Hermitian form $q$ on $\V=\bigotimes_{k=1}^d\V_k$ such that
\begin{equation*}
    q(\bz^{(1)}\otimes\ldots\otimes\bz^{(d)},\bw^{(1)}\otimes\ldots\otimes\bw^{(d)})=\prod_{k=1}^dq_k(\bz^{(k)},\bw^{(k)})
\end{equation*}
where $\bz^{(k)},\bw^{(k)}\in\V_k$ for $k\in[d]$, which is called the \emph{Bombieri-Weyl} or \emph{Frobenius} form. In particular, the Bombieri-Weyl form is unitary invariant in the sense that for any $\bz,\bw\in\V$ and unitary $g\in\U(\V_1)\times\ldots\times\U(\V_d)$ it holds $q(g\cdot\bz,g\cdot\bw)=q(\bz,\bw)$.

If we set a basis on each $\V_k$, we denote the space $\T^{\bn}=\bigotimes_{k=1}^d\C^{n_k}$ or simply $\T^n_d$ when $n_1=\ldots=n_d=n\in\N$, thus we identify $\T^{(n,m)}=\M_n^m$. Assume we have chosen orthonormal bases for which any $q_k$ is the canonical Hermitian product. Let $\mathcal{A}=(a_{j_1\ldots j_d}),\mathcal{B}=(b_{j_1\ldots j_d})\in\T^{\bn}$ be tensors, then the Bombieri-Weyl form is the Hermitian inner product
\begin{equation*}
\langle\mathcal{A},\mathcal{B}\rangle_{\C}=\sum_{j_1=1}^{n_1}\cdots\sum_{j_d=1}^{n_d}a_{j_1\ldots j_d}\bar{b}_{j_1\ldots j_d}.
\end{equation*}

\begin{definition}
Let $\mathcal{A}=(a_{j_1\ldots j_{d}})\in\T^{\bn}$ and $\mathcal{B}=(b_{j_1\ldots j_{e}})\in\T^{\bm}$ be tensors such that $\bm=(n_{k_1},\ldots,n_{k_{e}})$ for $1\leq k_1\leq\ldots\leq k_{e}\leq d$. Then, the \emph{contraction} $\mathcal{A}\times\mathcal{B}\in\T^{\bn\setminus\bm}$ where
\begin{equation*}
    \bn\setminus\bm\coloneqq(n_1,\ldots,n_{k_1-1},n_{k_1+1},\ldots,n_{k_2-1},n_{k_2+1},\ldots,n_{d})
\end{equation*}
of $\mathcal{A}$ and $\mathcal{B}$ is such that
\begin{equation*}
    (\mathcal{A}\times\mathcal{B})_{j_1\ldots j_{k_1-1}j_{k_1+1}\ldots j_{k_2-1}j_{k_2+1}\ldots j_{d}}=\sum_{j_{k_1}=1}^{n_{k_1}}\cdots\sum_{j_{k_{e}}=1}^{n_{k_{e}}}a_{j_1\ldots j_{d}}b_{j_{k_1}\ldots j_{k_{e}}}.
\end{equation*}
\end{definition}

\begin{remark}
Note that, if $\mathcal{A},\mathcal{B}\in\T^{\bn}$, then we have the equivalence $\mathcal{A}\times\bar{\mathcal{B}}=\langle\mathcal{A},\mathcal{B}\rangle_{\C}$.
\end{remark}

\begin{definition}
The \emph{spectral norm} in $\T^{\bn}$ is defined to be
\begin{equation*}
    \|\mathcal{A}\|_{S}\coloneqq\sup_{\bz^{(k)}\in\C^{n_k}\setminus\lbrace 0\rbrace}\frac{|\langle\mathcal{A},\otimes_{k=1}^d\bz^{(k)}\rangle_{\C}|}{\prod_{k=1}^d\|\bz^{(k)}\|_{2}}=\max_{\substack{\bz^{(k)}\in\C^{n_k} \\ \|\bz^{(k)}\|_2=1}}|\langle\mathcal{A},\otimes_{k=1}^d\bz^{(k)}\rangle_{\C}|\qquad\text{for $\mathcal{A}\in\T^{\bn}$.}
\end{equation*}
\end{definition}

The following notion is the generalization of Definition~\ref{defnormmat} in the case of tensors.

\begin{definition}
Let $\|\phantom{\bz}\|$ be the same norm on $\C^{n_1},\ldots,\C^{n_d}$, then the \emph{induced tensor norm} on $\T^{\bn}$ is defined to be
\begin{equation*}
    \|\mathcal{A}\|\coloneqq\sup_{\substack{j\in [d]\\\bz^{(k)}\in\C^{n_k}\setminus\lbrace 0\rbrace}}\frac{\left\|\mathcal{A}\times\otimes_{\substack{k=1\\ k\neq j}}^d\bz^{(k)}\right\|}{\prod_{\substack{k=1 \\ k\neq j}}^d\|\bz^{(k)}\|}\qquad\text{for $\mathcal{A}\in\T^{\bn}$.}
\end{equation*}
\end{definition}

\begin{proposition}\label{equalnorm}
In $\T^{\bn}$ there holds the equality of norms $\|\phantom{A}\|_{S}=\|\phantom{A}\|_{2}$.
\end{proposition}
\begin{proof}
In the case $d=2$ we have $\T^{\bn}=\M_{n_1}^{n_2}$. We have already studied this variety in Section~\ref{sec:matrices} and the assertion follows from the equalities
\begin{equation*}
    \|A\|_{2}=\sup_{\bz\in\C^m\setminus\lbrace 0\rbrace}\frac{\|A\bz\|_{2}}{\|\bz\|_{2}}=\sup_{\bz^{(k)}\in\C^{n_k}\setminus\lbrace 0\rbrace}\frac{|(\bz^{(1)})^{T}A\bz^{(2)}|}{\|\bz^{(1)}\|_{2}\|\bz^{(2)}\|_{2}}=\max_{\substack{\bz^{(k)}\in\C^{n_k} \\ \|\bz^{(k)}\|_2=1}}|(\bz^{(1)})^{T}A\bz^{(2)}|
\end{equation*} 
where we used the known fact
\begin{equation*}
    \|A\bz\|_2=\sup_{\bw\in\C^n\setminus\lbrace 0\rbrace}\frac{|\bw^{T}A\bz|}{\|\bw\|_{2}}=\max_{\substack{\bw\in\C^n\\ \|\bw\|_2=1}}|\bw^{T}A\bz|\qquad\text{for $A\in\M_n^m$ and $\bz\in\C^m$.}
\end{equation*}
Iterating this procedure the claim follows.
\end{proof}

\begin{remark}
Note that in general, as we have seen in the case of matrices in Section~\ref{sec:matrices}, the Hermitian norm $\|\phantom{A}\|_{\C}$ and the spectral norm $\|\phantom{A}\|_{S}$ do not coincide, see Corollary~\ref{disnorm}.
\end{remark}

%%%-----------------------------------------------------------------------------------------------

\subsection{Segre variety}\label{ssec:compsing}

Consider $\Seg\V\subseteq\V$ the Segre variety given by the Segre embedding 
\begin{align*}
    &\Pro\V_1\times\ldots\times\Pro\V_d\to\Seg\V\subseteq\Pro\V\\
    &\left([\bz^{(1)}],\ldots,[\bz^{(d)}]\right)\mapsto[\bz^{(1)}\otimes\ldots\otimes\bz^{(d)}]
\end{align*}
The Segre variety coincides with the variety of rank one, or equivalently decomposable, tensors. If we set bases on each $\V_k$, we denote $\Seg\T^{\bn}\subseteq\T^{\bn}$ the Segre variety and the Segre embedding reads
\begin{align*}
    &\Pro^{n_1-1}\times\ldots\times\Pro^{n_d-1}\to\Seg\T^{\bn}\subseteq\Pro^{n_1\cdots n_d-1}\\
    &\left([\bz^{(1)}],\ldots,[\bz^{(d)}]\right)\mapsto[z^{(1)}_{1}\cdots z^{(d)}_{1},\ldots,z^{(1)}_{j_1}\cdots z^{(d)}_{j_d},\ldots,z^{(1)}_{n_1}\cdots z^{(d)}_{n_d}]
\end{align*}

Let $X=\Seg\T^{\bn}$, for a tensor $\mathcal{U}\in\T^{\bn}$ we aim to solve
\begin{equation*}
  \min_{Z\in X}\|\mathcal{U}-Z\|_{\C}^2.
\end{equation*}
In order to solve the Euclidean version of this question
\begin{equation*}
  \min_{Z\in X}\|\mathcal{U}-Z\|_{\R}^2,
\end{equation*}
in \cite{lim2006} Lim proposes a variational approach on real tensors that links this problem to the problem of maximizing the function $\mathcal{U}\times Z=\langle\mathcal{U},Z\rangle_{\R}$ and which has led to various results. In particular, in the same work Lim introduces the notion of a \emph{singular vector tuple} of $\mathcal{U}\in\T^{\bn}$.

The same approach is then used again in \cite{of2014} where, more generally, the notion of singular vector tuple is linked to the solutions of the Euclidean critical ideal of a complex tensor. Here, we will consider this last  terminology for which a singular vector tuple is a $d$-tuple
\begin{equation*}
(\bz^{(1)},\ldots,\bz^{(d)})\in\C^{n_1}\times\ldots\times\C^{n_d}
\end{equation*}
of unit vectors such that for $k=1,\ldots,d$ there hold the equalities
\begin{equation}\label{singt}
\mathcal{U}\times(\otimes_{j\in[d]\setminus\lbrace k\rbrace}\bz^{(j)})=\langle\mathcal{U},\otimes_{j=1}^d\bz^{(j)}\rangle_{\R}\bz^{(k)}
\end{equation}
where the scalar $\langle\mathcal{U},\otimes_{j=1}^d\bz^{(j)}\rangle_{\R}$ is possibly non real and it is the associated \emph{singular value}. This way the singular vector tuples exactly characterize the solutions of the critical ideal of the $\ED$, see \cite{dhost2014}, using the formula $\langle\mathcal{U},\otimes_{j=1}^d\bz^{(j)}\rangle_{\R}\otimes_{k=1}^d\bz^{(k)}\in X$. What follows is the Hermitian version of this argument and can be found in literature for example in the work of Hiling and Sudbery \cite{hs2009}.

A rank one tensor $Z\in X$ can be written as $Z=\mu\, \bz^{(1)}\otimes\ldots\otimes\bz^{(d)}$ where $\mu\in\C$ and the vectors $\bz^{(k)}\in\C^{n_k}$ are unit vectors or in other terms satisfy $\|\bz^{(k)}\|^2_{\C}=1$. Thus, we obtain
\begin{align*}
  \min_{Z\in X}\|\mathcal{U}-Z\|_{\C}^2&=\min_{Z\in X}\left(\|\mathcal{U}\|_{\C}^2-\langle\mathcal{U},Z\rangle_{\C}-\overline{\langle\mathcal{U},Z\rangle_{\C}}+\|Z\|_{\C}^2\right)\\
  &=\min_{\substack{\|\bz^{(k)}\|^2_{\C}=1\\ \mu\in\C}}\left(\|\mathcal{U}\|_{\C}^2-\mu\langle\mathcal{U},\otimes_{k=1}^d\bz^{(k)}\rangle_{\C}-\bar{\mu}\overline{\langle\mathcal{U},\otimes_{k=1}^d\bz^{(k)}\rangle_{\C}}+|\mu|^2\right)\\
  &=\max_{\|\bz^{(k)}\|^2_{\C}=1}\frac{\left(\langle\mathcal{U},\otimes_{k=1}^d\bz^{(k)}\rangle_{\C}+\overline{\langle\mathcal{U},\otimes_{k=1}^d\bz^{(k)}\rangle_{\C}}\right)^2}{4}-\|\mathcal{U}\|_{\C}^2\\
  &=\max_{\|\bz^{(k)}\|^2_{\C}=1}\left|\langle\mathcal{U},\otimes_{k=1}^d\bz^{(k)}\rangle_{\C}\right|^2-\|\mathcal{U}\|_{\C}^2
\end{align*}
where we have chosen without loss of generality  $\mu=\langle\mathcal{U},\otimes_{k=1}^d\bz^{(k)}\rangle_{\C}$ to be real and nonnegative, see also below. Using the Lagrange multipliers we define the real-valued function 
\begin{equation*}
 L\colon\C^{n_1}\times\ldots\times\C^{n_d}\times\R^k\to\R
\end{equation*}
such that
\begin{align*}
 L(&\bz^{(1)},\ldots,\bz^{(d)},\lambda_1,\ldots,\lambda_k)\coloneqq\langle\mathcal{U},\otimes_{k=1}^d\bz^{(k)}\rangle_{\C}+\overline{\langle\mathcal{U},\otimes_{k=1}^d\bz^{(k)}\rangle_{\C}}-\sum_{k=1}^d\lambda_k(\|\bz^{(k)}\|_{\C}^2-1)\\
 &=\sum_{j_1=1}^{n_1}\cdots\sum_{j_d=1}^{n_d}(u_{j_1\ldots j_n}\bar{z}_{j_1}^{(1)}\cdots\bar{z}_{j_d}^{(d)}+\bar{u}_{j_1\ldots j_n}z_{j_1}^{(1)}\cdots z_{j_d}^{(d)})-\sum_{k=1}^d\lambda_k\left(\sum_{j=1}^{n_k}|z_{j}^{(k)}|^2-1\right)
\end{align*}
and its derivatives vanish if and only if for any $k\in[d]$ there hold the equalities
\begin{equation}\label{eqsing}
\mathcal{U}\times(\otimes_{j\in[d]\setminus\lbrace k\rbrace}\bar{\bz}^{(j)})=\lambda_{k}\bz^{(k)}\qquad\text{and}\qquad\|\bz^{(k)}\|_{\C}^2=1.
\end{equation}
Moreover, from the equalities
\begin{equation*}
    \lambda_k=\langle\lambda_k\bz^{(k)},\bar{\bz}^{(k)}\rangle_{\R}=\langle\mathcal{U}\times\otimes_{j\in[d]\setminus\lbrace k\rbrace}\bar{\bz}^{(j)},\bar{\bz}^{(k)}\rangle_{\R}=\langle\mathcal{U},\otimes_{j=1}^d\bz^{(j)}\rangle_{\C}
\end{equation*}
we deduce the independence of $\lambda_k$ from $k$. Thus, let $\lambda_k=\langle\mathcal{U},\otimes_{j=1}^d\bz^{(j)}\rangle_{\C}$ for any $k$. If this value along with the vectors $\bz^{(1)},\ldots,\bz^{(d)}$ satisfies the qualities \eqref{eqsing}, then for any $\xi\in\Sp^1\subseteq\C$ the value $\xi^{d}\langle\mathcal{U},\otimes_{j=1}^d\bz^{(j)}\rangle_{\C}$ along with the vectors $\bar{\xi}\bz^{(1)},\ldots,\bar{\xi}\bz^{(d)}$ satisfies the equalities \eqref{eqsing}. Thus, again we can assume that $\langle\mathcal{U},\otimes_{j=1}^d\bz^{(j)}\rangle_{\C}$ is real and nonnegative.

\begin{definition}
A \emph{Hermitian singular vector tuples} of $\mathcal{U}\in\T^{\bn}$ is a $d$-tuple
\begin{equation*}
(\bz^{(1)},\ldots,\bz^{(d)})\in\C^{n_1}\times\ldots\times\C^{n_d}
\end{equation*}
of unit vectors such that for any $k=1,\ldots,d$ it holds the equality
\begin{equation}\label{singvec}
\mathcal{U}\times(\otimes_{j\in[d]\setminus\lbrace k\rbrace}\bar{\bz}^{(j)})=\langle\mathcal{U},\otimes_{j=1}^d\bz^{(j)}\rangle_{\C}\bz^{(k)}
\end{equation}
and the scalar $\langle\mathcal{U},\otimes_{j=1}^d\bz^{(j)}\rangle_{\C}$ is real and non negative and it is the associated \emph{Hermitian singular value}.
\end{definition}

We discuss now the notion of Hermitian singular vector tuples for matrices, i.e.\ $d=2$.

\begin{example}
For a matrix $A\in\M_{n}^{m}$, the existence of a Hermitian singular vector tuple translates in the existence of two unit vectors $\bz^{(1)}\in\C^{n}$ and $\bz^{(2)}\in\C^{m}$ such that they satisfy a condition equivalent to the one of a singular pair $(\bz^{(1)},\bar{\bz}^{(2)})$ of $A$ since there hold the equalities
\begin{align*}
      A\bar{\bz}^{(2)}&=A\times\bar{\bz}^{(2)}=\langle A,\bz^{(1)}\otimes\bz^{(2)}\rangle_{\C}\bz^{(1)}=\left((\bz^{(1)})^HA\bar{\bz}^{(2)}\right)\bz^{(1)}\\
      A^H\bz^{(1)}&=\bar{A}\times\bz^{(1)}=\langle A,\bz^{(1)}\otimes\bz^{(2)}\rangle_{\C}\bar{\bz}^{(2)}=\left((\bz^{(1)})^HA\bar{\bz}^{(2)}\right)\bar{\bz}^{(2)}
\end{align*}
and the critical points of the Hermitian distance function are of the form 
\begin{equation*}
\langle A,\bz^{(1)}\otimes\bz^{(2)}\rangle_{\C}\ \bz^{(1)}\otimes\bz^{(2)}=\left((\bz^{(1)})^HA\bar{\bz}^{(2)}\right)\bz^{(1)}(\bar{\bz}^{(2)})^H.
\end{equation*}  
In other words, $(\bz^{(1)},\bz^{(2)})$ is a Hermitian singular vector tuples of $A$ if and only if $(\bz^{(1)},\bar{\bz}^{(2)})$ is a singular pair of $A$ and the value $(\bz^{(1)})^HA\bar{\bz}^{(2)}$ is both a Hermitian singular value and singular value of $A$ in the classical sense.
\end{example}

\begin{remark}
The singular values provided by equations \eqref{singt} do not coincide with classical singular values for non real matrices and non real vectors. Moreover, if a Hermitian singular vector tuples of a tensor $\mathcal{U}\in\T^{\bn}$ is real valued, then that is also a singular vector tuple of $\mathcal{U}$ and the relative Hermitian singular value is also a singular value. These considerations suggest that the Hermitian singular values could be a more natural concept to look at.

Note that, for a given Hermitian singular value $0\leq\mu\in\R$ with Hermitian singular vector tuple $(\bz^{(1)},\ldots,\bz^{(d)})$, then there exist at least $d$ Hermitian singular vector tuples of $\mu$ given by $(\xi\bz^{(1)},\ldots,\xi\bz^{(d)})$ where $\xi^d=1$. However, these all yield the same critical point.
\end{remark}

We obtain a direct consequence from the definition of Hermitian singular vector tuples.

\begin{corollary}
The tuple $(\bz^{(1)},\ldots,\bz^{(d)})$ is a Hermitian singular vector tuple of $\mathcal{U}\in\T^{\bn}$ iff 
\begin{equation*}
    \langle\mathcal{U},\otimes_{k=1}^d\bz^{(k)}\rangle_{\C}\otimes_{j=1}^d\bz^{(j)}\in\Seg\T^{\bn}
\end{equation*}
is a critical point of the Hermitian distance from $\Seg\T^{\bn}$.
\end{corollary}

The next proposition follows from the observation that one of the singular vector tuple is the maximizer of $|\langle\mathcal{U},\otimes_{k=1}^d\bz^{(k)}\rangle_{\C}|$.

\begin{proposition}
The biggest Hermitian singular values of a tensor $\mathcal{U}\in\T^{\bn}$ is $\|\mathcal{U}\|_{S}$.
\end{proposition}
\begin{proof}
The statement follows considering the maximization problem and any of the definition of the norm, see Proposition~\ref{equalnorm}.
\end{proof}

\begin{corollary}\label{disnorm}
In $\T^{\bn}$ it holds the inequality of norms $\|\phantom{U}\|_{S}\leq\|\phantom{U}\|_{\C}$.
\end{corollary}
\begin{proof}
Let $\mathcal{U}\in\T^{\bn}$ and take $Z=\|\mathcal{U}\|_{S}\otimes_{k=1}^d\bz^{(k)}\in\Seg\T^{\bn}$ the critical point associated to the largest Hermitian singular value of $\mathcal{U}$, the assertion follows from the evaluation
\begin{equation*}
    0\leq\|\mathcal{U}-Z\|_{\C}^2=\|\mathcal{U}\|_{\C}^2-\langle\mathcal{U},Z\rangle_{\C}-\overline{\langle\mathcal{U},Z\rangle_{\C}}+\|Z\|_{\C}^2=\|\mathcal{U}\|_{\C}^2-\|\mathcal{U}\|_{S}^2.
\end{equation*}
\end{proof}

There is an established general theory about complex singular vector tuple with vanishing singular value. Such tuples are solutions of the system 
\begin{equation}\label{d0tuple}
    \langle\mathcal{U},\otimes_{k=1}^d\bz^{(k)}\rangle_{\C}=0\quad\text{and}\quad\nabla_{\bar{\bz}^{(k)}}\langle\mathcal{U},\otimes_{k=1}^d\bz^{(k)}\rangle_{\C}=0\quad\text{for $k=1,\ldots,d$}.
\end{equation}
The number of independent equations of \eqref{d0tuple} is one more than the number of variables, so the variables can be eliminated to get a polynomial equation in the coefficients of $\mathcal{U}$.

\begin{theorem}[Cayley]
There exists a polynomial function $\hdet(\mathcal{U})$ in the coefficients of $\mathcal{U}$, such that the equations \eqref{d0tuple} have a solution with all $\bz^{(k)}$ nonzero for $k=1,\ldots,d$ if and only if $\hdet(\mathcal{U})=0$.
\end{theorem}

The function $\hdet$ is called the \emph{hyperdeterminant}. It is a special case of the discriminant of a function of several variables, see \cite{gkz1994}. The hyperdeterminant has appeared in the theory of multipartite entanglement, for example in \cite{m2003}.

\begin{remark}
The hyperderminant also arises in the Euclidean case since the set of tensors with zero as a Hermitian singular value is defined by the same equations \eqref{d0tuple} of the set of tensors with zero as a singular value.
\end{remark}

It is known that the dual variety of $\Seg\T^{\bn}$ is a hypersurface if and only if it holds the inequality
\begin{equation}\label{condhyp}
	2\max\lbrace n_1,\ldots,n_d\rbrace\leq n_1+n_2+\ldots+n_d-d+2,
\end{equation}
see \cite[Chapter XIV]{gkz1994}. In this case, the dual variety is defined by the hyperdeterminant. In particular, from \cite[Theorem 3.34]{f2025} and the fact that the varieties are all defined by real polynomials we get the following result.

\begin{corollary}
Let $\mathcal{U}$ be a tensor and $Z$ be a critical point of the Hermitian distance from $\Seg\T^{\bn}$, if inequality \eqref{condhyp} holds then $\hdet(\mathcal{U}-Z)=0$.
\end{corollary}

Let $X=\Seg\T_3^{2}\subseteq\T_3^{2}$. The Hermitian distance minimization problem is widely investigated on this variety, see for example \cite{wg2003}. In \cite{hs2009} the authors investigated the Hermitian distance problem on $X$, they proved that the square of a singular value of a tensors must be the zero of a polynomial of degree $12$, thus providing a bound for the number of critical points which in turn yields $\max\HD(X)\leq 12$. In the same work, for a subset of tensors in $\T_3^{2}$ the bound is refined to $5$ and the formulas to compute the Hermitian singular values along with the Hermitian singular vector tuples is provided. In particular, from that work, and the equality $\vHD(X)=8$ that we will see below, follows the inequality $\max\HD(X)\geq 6$.

\begin{proposition}
For $\Seg\T_3^{2}\subseteq\T_3^{2}$ there hold 
\begin{equation*}
    \vHD(\Seg\T_3^{2})=8\qquad\text{and}\qquad\lbrace 4,6,8\rbrace\subseteq\HD(\Seg\T_3^{2})\subseteq\lbrace 2,4,6,8\rbrace.
\end{equation*}
\end{proposition}
\begin{proof}
We will compute in Example~\ref{extens} the value of the $\vHD$. Now, let $\mathcal{U}\in\T_3^2$ be symmetric with $5$ Hermitian eigenvector (see Subsection~\ref{ssec:veronese}) which we know exists from Proposition~\ref{bin}. By Proposition~\ref{segrevero}, a Hermitian eigenvector $\bz\in\C^2$ generates the Hermitian singular vector tuple $(\bz,\bz,\bz)\in(\C^2)^3$ with the same Hermitian singular value and then this tensor possesses at least $5$ critical points. Since the number of solutions of the Hermitian critical ideal must be equal to $8$ there are $3$ more solutions. Moreover, since $\mathcal{U}$ is symmetric, any permutation of a Hermitian singular vector tuple yields another one. Thus, the only possible outcome is to have $3$ complex singular vectors that are of the form $(\bz^{(1)},\bz^{(1)},\bz^{(2)})$, $(\bz^{(1)},\bz^{(2)},\bz^{(1)})$, $(\bz^{(2)},\bz^{(1)},\bz^{(1)})$ for some $\bz^{(1)}\neq\bz^{(2)}\in\C^2$ and then this tensor possesses $3$ other critical points. In particular, it holds $\max\HD(\Seg\T_3^2)=8$. A similar reasoning provides tensors with $4$ and $6$ critical points, such examples can also be easily found by testing random points. For example the tensor
\begin{equation*}
	z_{000}=1,\ z_{010}=\frac{2001}{1000},\ z_{100}=2,\ z_{110}=3,\ z_{001}=2,\ z_{011}=3,\ z_{101}=3,\ z_{111}=5
\end{equation*}
admits $8$ critical points, the tensor
\begin{equation*}
	z_{000}=\frac{8}{5},\ z_{010}=2,\ z_{100}=\frac{1}{4},\ z_{110}=\frac{8}{9},\ z_{001}=\frac{4}{5},\ z_{011}=\frac{3}{2},\ z_{101}=\frac{5}{9},\ z_{111}=\frac{9}{7}
\end{equation*}
admits $6$ critical points, and the tensor
\begin{equation*}
	z_{000}=\frac{5}{2},\ z_{010}=\frac{8}{7},\ z_{100}=\frac{6}{7},\ z_{110}=\frac{2}{9},\ z_{001}=\frac{3}{8},\ z_{011}=\frac{7}{9},\ z_{101}=\frac{1}{2},\ z_{111}=\frac{9}{7}
\end{equation*}
admits $4$ critical points.
\end{proof}

\begin{example}\label{extens}
Let $X=\Seg\T_3^2$ defined by the ideal
\begin{align*}
  I_X=\langle&z_{111}z_{221}-z_{121}z_{211},\quad z_{112}z_{222}-z_{122}z_{212},\quad z_{111}z_{212}-z_{112}z_{211}\\
  &z_{121}z_{222}-z_{122}z_{221},\quad z_{111}z_{122}-z_{112}z_{121},\quad z_{211}z_{222}-z_{212}z_{221}\\
  &z_{121}z_{212}-z_{112}z_{221},\quad z_{111}z_{222}-z_{121}z_{212},\quad z_{111}z_{222}-z_{211}z_{122}\rangle.
\end{align*}
We can one-to-one parametrize a dense subvariety of $X$ as
\begin{align*}
    &\psi\colon\C^4\longrightarrow X\subseteq\T_3^2\\
    (z_1,&z_2,z_3,z_4)\mapsto\begin{bmatrix}
        z_1\\
        1
    \end{bmatrix}\otimes\begin{bmatrix}
        z_2\\
        1
    \end{bmatrix}\otimes\begin{bmatrix}
        z_3\\
        z_4
    \end{bmatrix}
\end{align*}
and the singular locus of this parametrization is $V(z_3,z_4)\subseteq\C^4$. Critical points satisfy the equations
\begin{equation*}
    \frac{\partial}{\partial z_k}\|\psi(\bz)-\bu\|_{\C}^2=0
\end{equation*}
for $k=1,\ldots,4$. Introducing $\bw$ and $\bv$ we compute that the Hermitian critical ideal is of degree $8$ and thus $\vHD(X)=8$.
\end{example}

%----------------------------------------------------------------------------------------------------------------------------------------------

\subsection{Veronese variety}\label{ssec:veronese}\thispagestyle{plain}

Let $n,d \in\N$ and let $\V_1$ be a complex vector space of dimension $n$ endowed with a Hermitian form $q_1$, in this section we consider $\V$ to be the $\binom{n-1+d}{d}$ dimensional vector subspace of symmetric $d$-tensors $\V\coloneqq\Sym^d\V_1\subseteq\bigotimes_{k=1}^d\V_1$. Similarly to the incipit of Section~\ref{sec:segre}, there exists a unique Hermitian form $q$ on $\V$ such that
\begin{equation*}
    q(\bz^{\otimes d},\bw^{\otimes d})=q_1(\bz,\bw)^d
\end{equation*}
where $\bz,\bw\in\V_1$ which is exactly the restriction on $\V$ of the Bombieri-Weyl form on $\bigotimes_{k=1}^d\V_1$. If we set a basis on $\V_1$, we denote $\St^n_d\subseteq\T_d^n$ the subspace of symmetric tensors that is the space of tensors $\mathcal{U}$ such that for any $\sigma$ permutation of $d$ elements it holds $\mathcal{U}_{j_1j_2\ldots j_d}=\mathcal{U}_{\sigma(j_1)\sigma(j_2)\ldots\sigma(j_d)}$. There exists a bijection
\begin{align*}
    &\St^n_d\to\C[\bz]_d\subseteq\C[\bz]\\
    &\mathcal{U}\mapsto\mathcal{U}\times\bz^{\otimes d}=\sum\limits_{\|\alpha\|_1=d}\binom{d}{\alpha}u_{\alpha}\bz^{\alpha}
\end{align*}
from $\St^n_d$ to the space of homogeneus polynomial of degree $d$ in $n$ variables $\C[\bz]_d$.

Consider also $v_d\V_1\subseteq\V$ the Veronese variety given by the Veronese embedding
\begin{align*}
    v_d\colon&\Pro\V_1\to v_d\V_1\subseteq\Pro\V\\
    &[\bz]\mapsto[\bz^{\otimes d}]
\end{align*}
The Veronese variety coincides with the variety of rank one symmetric $d$ tensors. If we set a basis, we denote $v_d\C^n\subseteq\St_d^{n}$ the Veronese variety and the Veronese embedding reads
\begin{align*}
    &\Pro^{n-1}\to v_d\C^n\subseteq\Pro^{\binom{n-1+d}{d}-1}\\
    &[\bz]\mapsto[z_1^d,\ z_1^{d-1}z_2,\ldots,z_n^{d-1}z_{n-1},\ z_n^d]
\end{align*}
The tensors in the Veronese variety are of the form $\mathcal{U}=(a_1z_1+a_2z_2+\ldots+a_nz_n)^d$ for some $a_1,\ldots,a_n\in\C$ not simultaneously vanishing.

Similarly to what happens for the Segre variety, in the Euclidean distance case the notion of \emph{eigenvector} and relative \emph{eigenvalue} of a tensor $\mathcal{U}\in\T_d^{n}$ is introduced. We can now generalize this notion. For a general $\mathcal{U}\in\T_d^{n}$ using the same reasoning of Subsection~\ref{ssec:compsing} the conditions on a Hermitian singular vector tuple becomes the one contained in the following definition.

\begin{definition}
A \emph{Hermitian eigenvector} of $\mathcal{U}\in\T_d^{n}$ is a unit vector $\bz\in\C^n$ such that it holds the equality
\begin{equation}\label{eigenvec}
\begin{bmatrix}
    \sum_{k=1}^d\mathcal{U}\times\left(\bar{\bz}^{\otimes(k-1)}\otimes e_1\otimes\bar{\bz}^{\otimes{d-k}}\right)\\
    \vdots\\
    \sum_{k=1}^d\mathcal{U}\times\left(\bar{\bz}^{\otimes(k-1)}\otimes e_n\otimes\bar{\bz}^{\otimes{d-k}}\right)\\
\end{bmatrix}=\langle\mathcal{U},\bz^{\otimes d}\rangle_{\C}\bz
\end{equation} 
and the scalar $\langle\mathcal{U},\bz^{\otimes d}\rangle_{\C}$ is real and non negative and it is the associated \emph{Hermitian eigenvalue}. Here $e_j$ for $j=1,\ldots,n$ is the $j$-th vector of the canonical basis of $\C^n$.

In particular, if $\mathcal{U}\in\St_d^{n}$, by symmetry the equation \eqref{eigenvec} simplifies to
\begin{equation}\label{eigenvecsym}
\mathcal{U}\times\bar{\bz}^{\otimes d-1}=\langle\mathcal{U},\bz^{\otimes d}\rangle_{\C}\bz.
\end{equation}
\end{definition}

\begin{remark}
If a Hermitian eigenvector of a tensor $\mathcal{U}\in\T^{\bn}$ is real valued, then that is also a eigenvector of $\mathcal{U}$ and the relative eigenvalue is also a Hermitian eigenvalue.
\end{remark}

\begin{corollary}
Let $\mathcal{U}\in\St_d^{n}$, then $\bz$ is a Hermitian eigenvector of $\mathcal{U}$ iff 
\begin{equation*}
    \langle\mathcal{U},\bz^{\otimes d}\rangle_{\C}\bz^d\in v_d\C^n
\end{equation*}
is a critical point of the Hermitian distance from $v_d\C^{n}$.
\end{corollary}

Considering the definitions of Hermitian singular vector tuple and Hermitian eigenvector through the introduction of the variables $\bw$ and $\bv$ we get the following result.

\begin{proposition}\label{segrevero}
Let $\mathcal{U},\mathcal{V}\in\St_d^{n}$, then the zero locus of the Hermitian critical ideal of $v_d\C^n\subseteq\T_d^n$ of $(\mathcal{U},\mathcal{V})$ is contained in the zero locus of the Hermitian critical ideal of $\Seg\T_d^n\subseteq\T_d^n$ of $(\mathcal{U},\mathcal{V})$. In particular, a Hermitian eigenvector $\bz\in\C^n$ of $\mathcal{U}$ generates a Hermitian singular vector tuple $(\bz,\ldots,\bz)\in\left(\C^n\right)^d$ of $\mathcal{U}$. Moreover, the polynomial $\vHDp_{v_d\C^n,\mathcal{U},\mathcal{V}}(t^2)$ divides the polynomial $\vHDp_{\Seg\T_d^n,\mathcal{U},\mathcal{V}}(t^2)$. In particular, there hold $\vHD(\Seg\T_d^n)\geq\vHD(v_d\C^n)$ and $\max\HD(\Seg\T_d^n)\geq\max\HD(v_d\C^n)$.
\end{proposition}

%%------------------------------------------------------------------------------------------------------

\subsection{Binary forms}

In this subsection we consider the case of binary forms, or equivalently the Veronese variety of Subsection~\ref{ssec:veronese} setting $n=2$.  In \cite{of2014} it is shown $\ED(v_d\C^2)=d$, more precisely in \cite[Theorem 1]{mm2017} the author shows that the number of real critical points of the Euclidean distance of a symmetric tensor is bounded from below by the number of real zeros of the polynomial representing it. From this we get that $\max\HD(v_d\C^2)\geq d$.

Let us recall that we identify a tensor $\mathcal{U}\in\St_d^2$ with the homogeneous polynomial 
\begin{equation*}
    \mathcal{U}=\sum_{k=0}^d\binom{d}{k}u_{k}z_1^{d-k}z_2^k\in\C[z_1,z_2].
\end{equation*}
In this context condition \eqref{eigenvecsym} can be expressed in the nice form of the bivariate generalized polynomial equation
\begin{align*}
    0&=\det\begin{bmatrix}
        \partial_{z_1}\overline{\mathcal{U}} & \bar{z}_1\\
        \partial_{z_2}\overline{\mathcal{U}} & \bar{z}_2
    \end{bmatrix}=\begin{bmatrix}
        \sum_{k=0}^{d-1}(d-k)\binom{d}{k}\bar{u}_{k}z_1^{d-k-1}z_2^k & \bar{z}_1\\
        \sum_{k=1}^{d}k\binom{d}{k}\bar{u}_{k}z_1^{d-k}z_2^{k-1} & \bar{z}_2
    \end{bmatrix}\\
    &=\left(\sum_{k=0}^{d-1}(d-k)\binom{d}{k}\bar{u}_{k}z_1^{d-k-1}z_2^k\right)\bar{z}_2-\left(\sum_{k=1}^{d}k\binom{d}{k}\bar{u}_{k}z_1^{d-k}z_2^{k-1}\right)\bar{z}_1.
\end{align*}
Since we consider an Hermitian eigenvector to be normalized, without loss of generality we divide this equation by $z_1^{d-1}\bar{z}_1$ and setting $z=\frac{z_2}{z_1}$ we obtain the univariate generalized polynomial equation
\begin{equation}\label{geneigen}
    \left(\sum_{k=0}^{d-1}(d-k)\binom{d}{k}\bar{u}_{k}z^k\right)\bar{z}-\sum_{k=1}^{d}k\binom{d}{k}\bar{u}_{k}z^{k-1}=0.
\end{equation}
The same reasoning could be applied to a tensor $\mathcal{U}\in\T_d^2$ and condition \eqref{eigenvec} in order to get a similar polynomial.

\begin{corollary}\label{corovd}
	Let $d>2$ and $\mathcal{U}\in\T_d^{2}$ be a generic tensor, it has at least $d-2$ and at most $5d-10$ Hermitian eigenvalues. In particular, $d-2\leq\min\HD(v_d\C^2)\leq\max\HD(v_d\C^2)\leq 5d-10$.
\end{corollary}
\begin{proof}
	Consider a binary form $\mathcal{U}$ such that $u_{d-1}\neq 0$. Then the generalized polynomial of equation \eqref{geneigen} is of the form of Proposition~\ref{poldeg} where the term with the highest degree is $d\bar{u}_{d-1}z^{d-1}\bar{z}$ and there must exist at least $(d-1)-1=d-2$ solutions. On the other hand, the upper bound $5d-10$ follows from \cite[Theorem 1]{kn2005}.  In a similar way the result can be derived for a generic tensor $\mathcal{U}\in\T_d^2$ considering the coefficient obtained using the same steps from condition \eqref{eigenvec}.
\end{proof}

\begin{proposition}\label{bin}
Let $\mathcal{U}\in\St_d^{2}$ be a generic binary form, then there exist at most $d^2-2d+2$ Hermitian eigenvalues of $\mathcal{U}$. More specifically, $\vHD(v_d\C^2)=d^2-2d+2$.
\end{proposition}
\begin{proof}
Let $p$ be the generalized polynomial of equation \eqref{geneigen}. Introducing the variable $w$ and a binary form $\mathcal{V}$, using the Bernstein–Kushnirenko theorem the number of solutions of the system $p(z,w)=\bar{p}(w,z)=0$ is bounded by the mixed volume $\MV_{2}(\triangle_1,\triangle_2)$ where the arguments are the convex hulls 
\begin{align*}
    \triangle_1&=\conv(\lbrace(k,j)\in\N^2\mid k\leq d-1,\ j\leq 1\rbrace),\\
    \triangle_2&=\conv(\lbrace(k,j)\in\N^2\mid k\leq 1,\ j\leq d-1\rbrace).
\end{align*}
Now, the use of the explicit formula of the mixed volume yields
\begin{align*}
    \MV_{2}(\triangle_1,\triangle_2)&=-\Vol_2(\triangle_1)-\Vol_2(\triangle_2)+\Vol_2(\triangle_1+\triangle_2)\\
    &-(d-1)-(d-1)+d^2=d^2-2d+2
\end{align*}
and hence, by the assumption of genericity, the assertion follows.
\end{proof}

\begin{proposition}\label{extenssym}
For the variety $v_3\C^2\subseteq\T_3^2$ there hold 
\begin{equation*}
    \vHD(v_3\C^2)=5\qquad\text{and}\qquad\HD(v_3\C^2)=\lbrace 1,3,5\rbrace.
\end{equation*}
\end{proposition}
\begin{proof}
The variety $v_3\C^2$ is defined by the sum of the ideal defining the variety $\Seg\T_3^2\subseteq\T_3^2$ (see Example~\ref{extens}) and the ideal
\begin{equation*}
    \langle z_{100}-z_{010},\ z_{100}-z_{001},\ z_{110}-z_{011},\ z_{110}-z_{101}\rangle.
\end{equation*}
By restricting ourselves to the subspace of symmetric tensors, we can simplify our problem and consider the variety $\hat{X}_3\subseteq\C^4\simeq\St_3^2$ we will see in Example~\ref{extwist} and the Hermitian form $q(\bz,\bw)=z_1\bar{w}_1+3z_2\bar{w}_2+3z_3\bar{w}_3+z_4\bar{w}_4$. Then, it is easy to see $\vHD(X)=5$ as predicted by Proposition~\ref{bin}. We used Macaulay2 to compute the following results. The tensor 
\begin{equation*}
    \mathcal{U}=z_1^3+3\cdot 2z_1^2z_2+3\cdot 3z_1z_2^2+5z_2^3\in\Sy_3^2
\end{equation*}
possesses $3$ real (Hermitian) eigenvectors with relative (Hermitian) eigenvalues given by
\begin{alignat*}{2}
    &\bz^{(1)}=\begin{bmatrix}
        -\frac{3}{\sqrt{13}}\\
        \frac{2}{\sqrt{13}}
    \end{bmatrix}&&\langle\mathcal{U},(\bz^{(1)})^{\otimes 3}\rangle_{\C}=\frac{1}{\sqrt{13}},\\
    &\bz^{(2)}=\begin{bmatrix}
        \frac{1}{\sqrt{\varphi^2+1}}\\
        \frac{\varphi}{\sqrt{\varphi^2+1}}
    \end{bmatrix}&&\langle\mathcal{U},(\bz^{(2)})^{\otimes 3}\rangle_{\C}=\frac{3\varphi^2+4\varphi+1}{\sqrt{\varphi^2+1}}=\frac{15+7\sqrt{5}}{\sqrt{10+2\sqrt{5}}},\\
    &\bz^{(3)}=\begin{bmatrix}
        -\frac{1}{\sqrt{\varphi^2-2\varphi+2}}\\
        \frac{\varphi-1}{\sqrt{\varphi^2-2\varphi+2}}
    \end{bmatrix}\qquad&&\langle\mathcal{U},(\bz^{(3)})^{\otimes 3}\rangle_{\C}=\frac{-3\varphi^2+10\varphi-8}{\sqrt{\varphi^2-2\varphi+2}}=\frac{-15+7\sqrt{5}}{\sqrt{10-2\sqrt{5}}},\\
\end{alignat*}
where $\varphi=\frac{1+\sqrt{5}}{2}$ is the golden ratio. These eigenvalues yield the critical points
\begin{equation*}
    \frac{1}{\sqrt{13}}(\bz^{(1)})^{\otimes 3},\qquad \frac{15+7\sqrt{5}}{\sqrt{10+2\sqrt{5}}}(\bz^{(2)})^{\otimes 3}\qquad\text{and}\qquad\frac{-15+7\sqrt{5}}{\sqrt{10-2\sqrt{5}}}(\bz^{(3)})^{\otimes 3}
\end{equation*}
respectively. Moreover, for $\mathcal{U}$ we find $2$ Hermitian eigenvectors with relative Hermitian eigenvalues that are approximated by
\begin{alignat*}{2}
   &\bz^{(4)}\approx\begin{bmatrix}
        \frac{1}{0.690474-1.01910i}\\
        \frac{-0.59375-0.0403436i}{0.690474-1.01910i}
    \end{bmatrix}&&\langle\mathcal{U},(\bz^{(4)})^{\otimes 3}\rangle_{\C}\approx 0.670012,\\
    &\bz^{(5)}=\bar{\bz}^{(4)}\approx\begin{bmatrix}
        \frac{1}{0.690474+1.01910i}\\
        \frac{-0.59375+0.0403436i}{0.690474+1.01910i}
    \end{bmatrix}\qquad&&\langle\mathcal{U},(\bz^{(5)})^{\otimes 3}\rangle_{\C}\approx 0.670012,
\end{alignat*}
which yields the approximate relative critical points
\begin{equation*}
    0.670012(\bz^{(4)})^{\otimes 3}\qquad\text{and}\qquad 0.670012(\bar{\bz}^{(4)})^{\otimes 3}.
\end{equation*}
Other tensors with $5$ critical points are \begin{equation*}
    25z_1^3+3\cdot 200z_1^2z_2+3\cdot 10z_1z_2^2+36z_2^3\quad\text{and}\quad 35z_1^3+3\cdot 12+z_1^2z_2+3\cdot 160z_1z_2^2+20z_2^3.
\end{equation*}
On the other hand is easy to find tensors with $1$ and $3$ critical points, for example take 
\begin{equation*}
    2z_1^3+3\cdot 2z_1^2z_2+3\cdot 10z_1z_2^2+2z_2^3\quad\text{and}\quad 50z_1^3+3\cdot 105z_1^2z_2+3\cdot 42z_1z_2^2+63z_2^3
\end{equation*}
respectively.
\end{proof}

The variety $v_d\C^2$ can be isometrically identified to the rational normal curve of degree $d$ we present below. Now, we treat this variety in presence of the canonical Hermitian inner product. Let $X_d$ be the moment curve of degree $d$ given by the parametrization
\begin{align*}
    \psi_d\colon&\C\to X_d\subseteq\C^d\\
    &z\mapsto(z,z^2,\ldots,z^d)
\end{align*}
and $\hat{X}_d\simeq v_d\C^2$ be the rational normal curve of degree $d$ given by the embedding
\begin{align*}
    v_d\colon&\Pro^{1}\to\hat{X}_d\subseteq\Pro^{d}\\
    &[\bz]\mapsto[z_1^d,z_1^{d-1}z_2,\ldots,z_2^d]
\end{align*}
The latter is the projective closure of the former. In \cite[Example 5.12]{dhost2014} it is proved that for generic symmetric bilinear forms it holds $\ED(\hat{X}_d)=3d-2$, while in the case of the inner product induced by the Bombier-Weyl form we have $\ED(\hat{X}_d)=d$ as stated in the beginning of this subsection. In particular, since $X_d$ is a parametrized variety of degree $d$ it holds the equality
\begin{equation*}
    \ED(X_d)^2=(2d-1)^2=d\ED(\hat{X}_d)+(d-1)^2.
\end{equation*}
We can also see the equality of the numbers at the extremes by noting that $\hat{X}_d=v_d([1,z_2])\cup v_d([z_1,1])$ and the intersection of this two subsets consist only in the point $[1,1]$. Thus, since the cones over $v_d([1,z_2])$ and $v_d([z_1,1])$ are both equal to $X_d$, for generic points we obtain $\ED(X_d)$ solutions for the critical ideal of the $\ED$ from points $[1,z_2]$ and $\ED(X_d)$ solutions for the critical ideal of the $\ED$ from points $[z_1,1]$ which in sum yields $\ED(X_d)^2$ solutions $(z_1,z_2)$ in the cone over $\hat{X}_d$. In the end, since the origin is a solution of degree $(d-1)^2$ of the critical ideal and the parametrization $v_d$ is $d$-to-one the result follows.

From \cite[Proposition 2.18]{f2025} follows that for generic Hermitian forms it holds the equality $\vHD(X_d)=2d^2-2d+1$. Moreover, a similar argument as above applies in the Hermitian case and we obtain the following result.

\begin{proposition}\label{rat}
Let $X_d\subseteq\C^d$ be the moment curve of degree $d$ and $\hat{X}_d\subseteq\Pro^d$ be the rational normal curve of degree $d$, then
\begin{equation*}
    \vHD(X_d)^2=d^2\vHD(\hat{X}_d)+(d-1)^4
\end{equation*}
and in particular $\vHD(\hat{X}_d)=3d^2-4d+2$.
\end{proposition}

\begin{example}\label{exratnorm2}
(Moment curve of degree $2$). This variety is a parabola and was studied in \cite[Theorem 3.19, Example 4.1]{f2025}, in particular it holds $\vHD(X)=5$ in accordance to Proposition~\ref{rat}..

(Rational normal curve of degree $2$). This variety was studied in \cite[Theorem 3.18, Example 4.1]{f2025}, in particular it holds $\vHD(\hat{X}_2)=6$ in accordance to Proposition~\ref{rat}. The polynomial of the vHD discriminant when $\bv=\bu$ is
\begin{align*}
  \bigl(&u_1^4u_2^2+8u_1^4u_3^2-2u_1^2u_2^2u_3^2+20u_1^3u_3^3+8u_1^2u_3^4+u_2^2u_3^4\bigr)^3\\
  &\times\bigl(32u_1^6+435u_1^4u_2^2+384u_1^2u_2^4+256u_2^6-240u_1^5u_3-960u_1^3u_2^2u_3-960u_1u_2^4u_3\\
  &\qquad+696u_1^4u_3^2+1098u_1^2u_2^2u_3^2+384u_2^4u_3^2-980u_1^3u_3^3-960u_1u_2^2u_3^3+696u_1^2u_3^4\\
  &\qquad+435u_2^2u_3^4-240u_1u_3^5+32u_3^6\bigr).
\end{align*}
Note that the dual variety is $(\hat{X}_2)^{\vee}=V(4z_1z_3-z_2^2)\subseteq\C^3$, and thus $\hat{X}_2$, are partially solved by \cite[Theorem 3.18, Theorem 3.23]{f2025}.
\end{example}

\begin{example}\label{extwist}
(Moment curve of degree $3$). This variety is one-to-one parametrized by
\begin{align*}
\psi_3\colon&\C\to X_3=V(z_3-z_1^3,\ z_2-z_1^2)\subseteq\C^3\\
&z\mapsto\quad(z,\ z^2,\ z^3)
\end{align*}
Critical points satisfy the equation
\begin{equation*}
    \partial_{z}\|\psi_3(z)-\bu\|_{\C}^2=(\bar{z}-\bar{u}_1)+2z(\bar{z}^2-\bar{u}_2)+3z^2(\bar{z}^3-\bar{u}_3)=0.
\end{equation*}
Introducing $\bw,\bv$ we compute $\vHD(X_3)=13$ in accordance to Proposition~\ref{rat}. The diagonal of the matrix representing the Hermitian Killing form of the system
\begin{equation*}
    \begin{cases}
        p(z,w)=(z-u_1)+2w(z^2-u_2)+3w^2(z^3-u_3)=0\\
        \bar{p}(w,z)=(w-\bar{u}_1)+2z(w^2-\bar{u}_2)+3z^2(w^3-\bar{u}_3)=0
    \end{cases}
\end{equation*}
which generates the Hermitian critical ideal when $\bv=\bar{\bu}$, with respect to the basis 
\begin{equation*}
    \lbrace [1],[z],[w],[z^2],[zw],[w^2],[z^2w],[zw^2],[w^3],[z^2w^2],[zw^3],[w^4],[zw^4]\rbrace
\end{equation*}
is $\begin{bmatrix}
    13 & -4 & -4 & -\frac{4}{3} & -\frac{4}{3} & -\frac{4}{3} & a & a & a & b & b & b & c
\end{bmatrix}$ where
\begin{align*}
    &a=\frac{81|u_3|^2+20}{9},\\ &b=\frac{96|u_2|^2-432|u_3|^2-28}{27},\\
    &c=\frac{45|u_1|^2-480|u_2|^2+810|u_3|^2-4}{81}.
\end{align*}
In particular, the Rayleigh quotient of this matrix possesses a negative minimum and from \cite[Corollary 3.14]{f2024} we obtain $\max\HD(X_3)\leq 13-2=11$.

(Rational normal curve of degree $3$ / Twisted cubic). The affine cone of the twisted cubic is three-to-one parametrized by
\begin{align*}
\hat{\psi}_3&\colon\C^2\longrightarrow\hat{X}_3=V(z_1z_3-z_2^2,\ z_2z_4-z_3^2,\ z_1z_4-z_2z_3)\subseteq\C^4\\
(&z_1,z_2)\mapsto\quad (z_1^3,\ z_1^2z_2,\ z_1z_2^2,\ z_2^3)
\end{align*}
Critical points satisfy the equations
\begin{equation*}
    \partial_{z_1}\|\hat{\psi}_3(z_1,z_2)-\bu\|_{\C}^2=\partial_{z_2}\|\hat{\psi}_3(z_1,z_2)-\bu\|_{\C}^2=0.
\end{equation*}
Introducing $\bw,\bv$ we get a zero dimensional system of degree $13^2$. Since the origin is a solution of degree $4^2$ and the parametrization is three-to-one we get $\vHD(\hat{X}_3)=(13^2-4^2)/3^2=17$ in accordance to Proposition~\ref{rat}.
\end{example}

%-------------------------------------------------------------------------------------------

\section{Conclusion}
We have used the fundaments of the Hermitian Distance degree along with other machinery to compute the number of critical points of some important tensors varieties. This paper highlights the possibility given by the line of research.

%-------------------------------------------------------------------------------------------

\section*{Acknowledgments}
I would like to thank Prof.\ Giorgio Ottaviani for directing my research during the Ph.D. which culminates in this work. I also want to thank Nikhil Ken for pointing out the upperbound in Corollary~\ref{corovd}.

%-------------------------------------------------------------------------------------------

\appendix
\renewcommand{\thesection}{\Alph{section}} % corrected redefinition of '\thesection'
\makeatletter
\renewcommand\@seccntformat[1]{\appendixname\ \csname the#1\endcsname.\hspace{0.5em}}
\makeatother

\section{Topological degree of generalized polynomials}\label{sec:app}

In this section we relate topological properties of the generalized polynomials with the cardinality of their set of zeros. We start by recalling some classical results that can be found for example in \cite[Section 2.2]{h2001}. Denote with $H_n(\Sp^{n})$ the $n$-th homology group of the sphere $\Sp^{n}$. Let $f\colon\Sp^{n}\to\Sp^{n}$ be a continuous map, the \emph{topological degree} or simply \emph{degree} of $f$ is the number $\deg^{\prime}f\in\N$ such that the induced pushforward in homology $f_{\ast}\colon H_n(\Sp^{n})\to H_n(\Sp^{n})$ is the multiplication by $\deg^{\prime}f$. Let $f,g\colon\Sp^{n}\to\Sp^{n}$ then:
\begin{enumerate}%[\normalfont i)]
    \item $\deg^{\prime}\left(f\circ g\right)=\deg^{\prime}f\cdot\deg^{\prime}g$.
    \item $f$ and $g$ are homotopic if and only if it holds the equality $\deg^{\prime}f=\deg^{\prime}g$.
\end{enumerate}

Suppose now $f$ has the property that for some point $w\in\Sp^n$, the preimage $f^{-1}(w)$ consists of finitely many points, say $z^{(1)},\ldots,z^{(d)}$ and let $U_1,\ldots,U_d$ be mutually disjoint neighborhood of those points respectively. For any $k\in[d]$ the \emph{local degree} of $f$ at $z^{(k)}$ is the number $\deg^{\prime}f\arrowvert_{z^{(k)}}\in\N$ such that the pushforward in homology of $f\colon U_k\setminus\lbrace z^{(k)}\rbrace\to f(U_k)\setminus\lbrace w\rbrace$ is the multiplication by $\deg^{\prime}f\arrowvert_{z^{(k)}}$. In particular:
\begin{enumerate}
    \setcounter{enumi}{2}
    \item It holds the equality $\deg^{\prime}f=\sum_{k=1}^d\deg^{\prime}f\arrowvert_{z^{(k)}}$.
    \item If $f$ maps homeomorphically $U_k$ into $f(U_k)$ then it holds $\deg^{\prime}f\arrowvert_{z^{(k)}}=\pm 1$ depending on the orientation induced by the mapping of $f$.
\end{enumerate}

It is well known that a univariate polynomial $p(z)$ can be extended to a continuous map from the sphere $\Sp^2$ to itself and that it holds $\deg^{\prime}p=\deg p$. We now apply these concepts to generalized polynomials. Some of the following results are known, however they are not very present in the literature and could be useful to recall them. We start by proving a basic lemma about the conjugation map.

\begin{lemma}\label{conjsp}
The conjugation map $\bar{\phantom{z}}$ can be interpreted as a continuous map from $\Sp^2$ to itself such that $\deg^{\prime}\bar{\phantom{z}}=-1$.
\end{lemma}
\begin{proof}
Consider $\Sp^2=\C\cup\lbrace\infty\rbrace$ where $\infty$ is the point at infinity. We just define the map to be classical conjugation on $\C$ and to send $\infty$ to itself, in particular this map is orientation reversing and by properties 3.\ and 4.\ of the topological degree the claim follows.
\end{proof}

The next result characterizes the set of non constant generalized polynomials that can be extended as maps on the sphere $\Sp^2$.

\begin{proposition}\label{lemconj1}
A generalized polynomial $p(z,\bar{z})$ is constant or such that 
\begin{equation*}
    \lim_{|z|\rightarrow\infty}|p(z,\bar{z})|=\infty
\end{equation*}
if and only if it can be interpreted as a continuous map from $\Sp^2$ to itself.
\end{proposition}
\begin{proof}
Consider $\Sp^2=\C\cup\lbrace\infty\rbrace$ where $\infty$ is the point at infinity. For the only if part we define the map to be equal to $p$ on $\C$ and to send $\infty$ to itself if $\deg p\geq 1$ and to $p$ otherwise, i.e.\ if $\deg p=0$ we can regard $p$ as a constant. We need to check the continuity for the case $\deg p\geq 1$. For the open subset $\C$ we take the identity as coordinate map and obtain $\Id\circ p\circ\Id=p$. For a neighborhood of $\infty$ not containing the origin a coordinate map $\psi$ is given by $z\mapsto 1/z$ and the composition
\begin{equation*}
    \psi\circ p\circ\psi=\frac{1}{p(1/z,1/\bar{z})}
\end{equation*}
can be continuously extended since by hypothesis $\lim_{z\to 0}|p(1/z,1/\bar{z})|=\infty$ and thus $\lim_{z\to 0}\psi\circ p\circ\psi=0$.

For the contrary assume $p$ is non constant, otherwise we have finished. Moreover, assume that the limit of the statement does not hold. We prove the limit is thus non existent and then the map can not be continuously extended obtaining a contradiction. Collect the monomial with same degree and rewrite the generalized polynomial as
\begin{equation*}
    p(z,\bar{z})=\sum_{s=0}^{\deg p}\left(z^{s}\sum_{\ell=0}^{s}a_{s-\ell,\ell}\left(\frac{\bar{z}}{z}\right)^{\ell}\right).
\end{equation*}
Now, use the polar coordinates $z=\rho e^{i\theta}$ to write
\begin{equation}\label{eqiff}
    p(\rho e^{i\theta},\rho e^{-i\theta})=\sum_{s=0}^{\deg p}\left(\rho^{s}e^{si\theta}\sum_{\ell=0}^{s}a_{s-\ell,\ell}(e^{-2i\theta})^{\ell}\right).
\end{equation}
Note that the polynomial 
\begin{equation*}
    q(\lambda)=\sum_{\ell=0}^{\deg p}a_{\deg p-\ell,\ell}\lambda^{\ell}
\end{equation*}
has a finite number of solutions on $\Sp^1\subseteq\C$. In particular, the modulus of the coefficient of the monomial $\rho^{\deg p}$ of the polynomial in $\rho$ of equation \eqref{eqiff}, that is 
\begin{equation*}
    \left|e^{\deg pi\theta}\sum_{\ell=0}^{\deg p}a_{s-\ell,\ell}(e^{-2i\theta})^{\ell}\right|=|q(e^{-2i\theta})|,
\end{equation*}
is different from $0$ for some $\theta_1\in[0,2\pi]$. From this, it holds the limit 
\begin{equation*}
    \lim_{\rho\to\infty}|p(\rho e^{i\theta_1},\rho e^{-i\theta_1})|=\lim_{\rho\to\infty}\rho^{\deg p}|q(e^{-2i\theta_1})|=\infty
\end{equation*}
and hence the limit of the statement has to be undefined.
\end{proof}

In particular, a generalized polynomial $p(z,\bar{z})$ is orientation preserving or orientation reversing at a point $z_0$ if the determinant
\begin{equation*}
    \det\begin{bmatrix}
        (\partial_zp)(z_0,\bar{z}_0) & (\partial_{\bar{z}}p)(z_0,\bar{z}_0)\\
        (\partial_z\bar{p})(\bar{z}_0,z_0) & (\partial_{\bar{z}}\bar{p})(\bar{z}_0,z_0)
    \end{bmatrix}=\left|(\partial_zp)(z_0,\bar{z}_0)\right|^2-\left|(\partial_{\bar{z}}p)(z_0,\bar{z}_0)\right|^2
\end{equation*}
is positive or negative respectively. On the other hand, $z_0$ is said to be \emph{singular} if this determinant vanishes.

By topological arguments, if $\lim_{|z|\rightarrow\infty}|p(z,\bar{z})|=\infty$ then $p$ has zeros inside a compact subset, however the contrary is not valid. For example, consider the polynomial $p(z,\bar{z})=z+\bar{z}+i$ which has no zeros and the sequence $\lbrace b_k=ki\rbrace_{k\in\N}$ such that $|b_k|\rightarrow\infty$ for $k\to\infty$. Then, it is easy to compute
\begin{equation*}
    \lim_{k\to\infty}p(b_k,\bar{b}_k)=i.
\end{equation*} 

With the next result we compute the degree of a monomial on the sphere.

\begin{corollary}\label{degnm}
The topological degree of the continuous map $z^n\bar{z}^m\colon\Sp^2\to\Sp^2$ is $n-m$.
\end{corollary}
\begin{proof}
We let $n\geq m$, the cases $n<m$ follows from Lemma~\ref{conjsp} since $\deg^{\prime}z^n\bar{z}^m=-\deg^{\prime}\bar{z}^nz^m$. By noting that the map $z^n\bar{z}^m=|z|^{2m}z^{n-m}$ is a local homeomorphism on $\C\setminus\lbrace 0\rbrace$ which is orientation preserving and that the cardinality of the preimage of any element in $\C\setminus\lbrace 0\rbrace$ is equal to $n-m$, the assertion follows by properties 3.\ and 4.\ of the topological degree.
\end{proof}

The following result links the topological degree of a generalized polynomial on the sphere to the number of its zeros.

\begin{proposition}\label{lbdeg}
The number of zeros of a generic generalized polynomial $p(z,\bar{z})$ that satisfies the hypothesis of Proposition~\ref{lemconj1} is bounded from below by $|\deg^{\prime}p|$. 
\end{proposition}
\begin{proof}
If $p$ has an infinite number of zeros the statement is trivial. If not, since the map induced by a generic generalized polynomial is locally a homeomorphism in some neighborhoods of the preimages of $0$, the assertion follows from the properties 3.\ and 4.\ of the topological degree by the inequalities 
\begin{equation*}
    \sum_{z\in p^{-1}(0)}1=\sum_{z\in p^{-1}(0)}\left|\deg^{\prime}p\arrowvert_{z}\right|\geq\left|\sum_{z\in p^{-1}(0)}\deg^{\prime}p\arrowvert_{z}\right|=\left|\deg^{\prime}p\right|.
\end{equation*}
\end{proof}

The genericity assumption in the last proposition is essential to say that the generalized polynomial is a local homeomoprhism on the preimage of $0$. As a counterexample we can take $z^{n}\bar{z}^{m}$ that has $0$ as the only zero and can have arbitrary topological degree.

We compute the degree of a family of generalized polynomials that are useful to us.

\begin{proposition}\label{poldeg}
Let $p(z,\bar{z})=\sum a_{k,j}z^k\bar{z}^j$ be a generalized polynomial with $a_{n,m}z^n\bar{z}^m$ the only monomial of $p$ such that $n+m=\deg p$, then $\deg^{\prime}p=n-m$. In particular, the equation $p=0$ generically admits at least $|n-m|$ solutions.
\end{proposition}
\begin{proof}
The continuous homotopy
\begin{align*}
    F\colon&\C\times[0,1]\to\C\\
    &(z,t)\mapsto tp+(1-t)a_{n,m}z^n\bar{z}^m=t\sum_{k+j<\deg p}a_{k,j}z^k\bar{z}^j+a_{n,m}z^n\bar{z}^m
\end{align*}
of $p$ and $a_{n,m}z^n\bar{z}^m$ satisfies the hypothesis of Proposition~\ref{lemconj1} uniformly in $t$, then the homotopy can be extended on $\Sp^2$ and thus $\deg^{\prime}p=\deg^{\prime}a_{n,m}z^n\bar{z}^m=n-m$ by property 2.\ of the topological degree and Corollary~\ref{degnm}. The last part follows from Proposition~\ref{lbdeg}.
\end{proof}

The next result provides a sufficient condition for a generalized polynomial to be extended on the sphere.

\begin{lemma}\label{limitlem}
Let $p(z,\bar{z})=\sum a_{k,j}z^k\bar{z}^j$ be a non constant generalized polynomial, if the polynomial 
\begin{equation*}
    q(\lambda)=\sum_{\ell=0}^{\deg p}a_{\deg p-\ell,\ell}\lambda^{\ell}
\end{equation*}
has no zeros in $\Sp^1$ then $\lim_{|z|\rightarrow\infty}|p(z,\bar{z})|=\infty$.
\end{lemma}
\begin{proof}
Using the polar coordinates $z=\rho e^{i\theta}$, by the same steps of the second part of the proof of Proposition~\ref{lemconj1}, the limit in the statement reads
\begin{equation*}
    \lim_{\rho\rightarrow\infty}|p(\rho e^{i\theta},\rho e^{-i\theta})|=\lim_{\rho\rightarrow\infty}\left|\sum_{s=0}^{\deg p}\left(\rho^{s}e^{si\theta}\sum_{\ell=0}^{s}a_{s-\ell,\ell}(e^{-2i\theta})^{\ell}\right)\right|=\infty
\end{equation*}
uniformly for $\theta\in[0,2\pi]$. Since the module of the quantity multiplying $\rho^{\deg p}$ is $|q(e^{-2i\theta})|$ which by hypothesis is bounded from below by a positive constant independent from $\theta$ this limit holds.
\end{proof}

The contrary of Lemma~\ref{limitlem} is not valid, for example consider the polynomial $p(z,\bar{z})=z^2+\bar{z}^2+z$ for which it holds 
\begin{equation*}
    |p(z,\bar{z})|^2=p^{\Re}(z,\bar{z})^2+p^{\Im}(z,\bar{z})^2=\left(z^2+\bar{z}^2+\frac{z+\bar{z}}{2}\right)^2+\left(\frac{z-\bar{z}}{2i}\right)^2
\end{equation*}
and thus $\lim_{|z|\rightarrow\infty}|p(z,\bar{z})|=\infty$. In this case, the polynomial $q(\lambda)=a_{2,0}+a_{0,2}\lambda=1+\lambda$ of Lemma~\ref{limitlem} vanishes for $\lambda=-1\in\Sp^1$. 

\begin{remark}
The condition of Lemma~\ref{limitlem} is vacuously satisfied if there is only one leading term. While, if there are two leading terms, the condition is equivalent to have leading coefficients with different norms. Moreover, from the proof of \cite[Lemma 4.5]{f2024}, it follows that the condition of Lemma~\ref{limitlem} is also necessary if the generalized polynomial is of degree one.
\end{remark}

%-------------------------------------------------------------------------------------------

\bibliographystyle{elsarticle-num}
\bibliography{references}

@article{m2003,
  title =	 {Classification of multipartite entangled states by multidimensional determinants},
  author =	 {Akimasa Miyake},
  journal =	 {Physical Review A},
  volume =	 67,
  doi =		 {10.1103/PhysRevA.67.012108},
  year =	 2003,
  pages =	 {012108},
}

@book{h2001,
  author =	 {Allen Hatcher},
  title =	 {Algebraic {T}opology},
  publisher =	 {Cambridge University Press},
  year =	 2002}

@book{gkz1994,
  author =	 {Israel M. Gelfand and Mikhail M. Kapranov and Andrei V. Zelevinsky},
  title =	 {Discriminants, resultants and higher-dimensional determinants},
  publisher =	 {Birkh\"auser},
  year =	 1994}

@article{oss2014,
  title =	 {Exact Solutions in {S}tructured {L}ow-{R}ank {A}pproximation},
  author =	 {Giorgio Ottaviani and Pierre-Jean Spaenlehauer and Bernd Sturmfels},
  doi =		 {10.1137/13094520X},
  journal =	 {SIAM Journal on Matrix Analysis and Applications},
  number =   4,
  volume =	 35,
  year =	 2014,
  month =	 dec,
  pages =	 {1521--1542},
}

@article{wg2003,
  title =	 {Geometric measure of entanglement and applications to bipartite and multipartite quantum states},
  author =	 {Tzu-Chieh Wei and Paul M. Goldbart},
  doi =		 {10.1103/PhysRevA.68.042307},
  journal =	 {Physical Review A},
  volume =	 68,
  year =	 2003,
  month =	 oct,
  pages =	 {042307},
}

@inproceedings{lim2006,
    author = {Lek-Heng Lim},
    title = {Singular values and eigenvalues of tensors: {A} variational approach},
    booktitle = {Proceedings of the IEEE International Workshop on Computational Advances in Multi-Sensor Adaptive Processing},
    doi =		 {10.48550/arXiv.math/0607648},
    volume = 1,
    year = 2005,
    pages =	 {129--132},
}

@inproceedings{kn2005,
    author = {Dmitry Khavinson and Genevra Neumann},
    title = {On the number of zeros of certain rational harmonic functions},
    volume = 134,
    booktitle = {Proceedings of the American Mathamatical Society},
    doi =		 {10.48550/arXiv.math/0401188},
    year = 2006,
    month = apr,
    pages =	 {1077--1085},
}

@article{dh2016,
  title =	 {The average number of critical rank-one approximations to a tensor},
  author =	 {Jan Draisma and Emil Horobe\c{t}},
  doi =		 {10.1080/03081087.2016.1164660},
  journal =	 {Linear and Multilinear Algebra},
  volume =	 64,
  issue =    12,
  year =	 2016,
  pages =	 {2498--2518},
}

@article{dhost2014,
  title =	 {The {E}uclidean {D}istance {D}egree of an {A}lgebraic {V}ariety},
  author =	 {Jan Draisma and Emil Horobe\c{t} and Giorgio Ottaviani and Bernd Sturmfels and Rekha R. Thomas},
  doi =		 {10.1007/s10208-014-9240-x},
  journal =	 {Foundations of Computational Mathematics},
  volume =	 16,
  year =	 2015,
  pages =	 {99--149},
}

@article{hs2009,
  title =	 {The geometric measure of multipartite entanglement and the singular values of a hypermatrix},
  author =	 {Joseph J. Hilling and Anthony Sudbery},
  doi =		 {10.1063/1.3451264},
  journal =	 {Journal of Mathematical Physics},
  volume =	 51,
  year =	 2010,
  pages =	 {072102},
}

@article{f2025,
  title =	 {The {H}ermitian {D}istance degree of an {A}lgebraic {V}ariety},
  author =	 {Davide Furch\`i},
  doi =		 {10.48550/arXiv.2510.19461},
  journal =	 {arXiv},
  year =	 {2026},
}

@article{f2024,
  title =	 {The {H}ermitian {K}illing form and root counting of complex polynomials with conjugate variables},
  author =	 {Davide Furch\`i},
  doi =		 {10.1016/j.laa.2024.11.028},
  journal =	 {Linear Algebra and its Applications},
  volume =	 708,
  year =	 2025,
  pages =	 {93--111},
}

@article{p2002,
  title =	 {The (matrix) discriminant as a determinant},
  author =	 {Beresford N. Parlett},
  doi =		 {10.1016/S0024-3795(02)00335-X},
  journal =	 {Linear Algebra and its Applications},
  volume =	 335,
  issue =    {1--3},
  year =	 2002,
  pages =	 {85--101},
}

@article{mm2017,
  title =	 {The number of real eigenvectors of a real polynomial},
  author =	 {Mauro Maccioni},
  doi =		 {10.1007/s40574-016-0112-y},
  journal =	 {Bollettino dell'Unione Matematica Italiana},
  volume =	 11,
  issue =    {1--3},
  year =	 2018,
  pages =	 {125--145},
}

@article{of2014,
  title =	 {The number of singular vector tuples and uniqueness of best rank one approximation of tensors},
  author =	 {Shmuel Friedland and Giorgio Ottaviani},
  doi =		 {10.1007/s10208-014-9194-z},
  journal =	 {Foundations of Computational Mathematics},
  volume =	 14,
  issue =    {1--3},
  year =	 2014,
  pages =	 {1209--1242},
}
\end{document}